\documentclass{amsart}
\usepackage{amssymb}
\usepackage{amsfonts}

\newtheorem{theorem}{Theorem}
\theoremstyle{plain}

\newtheorem{definition}{Definition}
\newtheorem{example}{Example}

\newtheorem{lemma}{Lemma}

\newtheorem{proposition}{Proposition}
\newtheorem{remark}{Remark}

\numberwithin{equation}{section}
\input{tcilatex}

\begin{document}
\title{Characterizations of LA-semigroups by new types of fuzzy ideals }
\author{Muhammad Aslam}
\address{Deparment of Mathematics, Quaid-i-Azam University 45320, Islamabad
44000, Pakistan}
\email{draslamqau@yahoo.com}
\author{S. Abdullah}
\address{Deparment of Mathematics, Quaid-i-Azam University 45320, Islamabad
44000, Pakistan}
\email{saleemabdullah81@yahoo.com}
\author{M. Atique Khan}
\address{Deparment of Mathematics, Quaid-i-Azam University 45320, Islamabad
44000, Pakistan}
\email{atiquekhan\_16@yahoo.com}

\begin{abstract}
In this article, we give some characterization results of $\left( \in
_{\gamma },\in _{\gamma }\vee q_{\delta }\right) $-fuzzy left(right) ideals, 
$\left( \in _{\gamma },\in _{\gamma }\vee q_{\delta }\right) $-fuzzy
generalized bi-ideals and $\left( \in _{\gamma },\in _{\gamma }\vee
q_{\delta }\right) $-fuzzy bi-ideals of an LA-semigroup. We also give some
characterizations of LA-semigroups by the properties of $\left( \in _{\gamma
},\in _{\gamma }\vee q_{\delta }\right) $-fuzzy ideals.
\end{abstract}

\maketitle

\section{Introduction}

Kazim and Naseerudin have defined the concept of an LA-semigroup in 1972 
\cite{qw,MN}. Let $S$ be a non-empty set and binary operation $\ast $
satisfies $(a\ast b)\ast c=(c\ast b)\ast a$ for all $a,b,c\in S$. Then, $S$
is called an LA-semigroup. Later, Mushtaq et al. have investigated the
structure further and added many useful results to theory of LA-semigroups
see \cite{QMK,QMY,QMMY}. Mushtaq and Khan have developed ideal theory of
LA-semigroups \cite{MM}. In \cite{MNA}, Khan and Ahmad have given some
characterizations of LA-semigroups by their ideals. L. A. Zadeh founder of
the concept of fuzzy set \cite{1} in 1965. Behalf of this concept,
mathematicians initiated a natural framework for generalizing some basic
notions of algebra, e.g group theory, set theory, ring theory, topology,
measure theory and semigroup theory etc. Fuzzy logic is very useful in
modeling the granules as fuzzy sets. Bargeila and Pedrycz considered this
new computing methodology in \cite{22}. Pedrycz and Gomide in \cite{23}
considered the presentation of update trends in fuzzy set theory and its
applications. Foundations of fuzzy groups were laid by Rosenfeld in \cite{2}%
. The concept of fuzzy semigroup was introduced by Kuroki \cite{3,4}.
Recently, Khan and Khan defined the concept of fuzzy ideals in LA-semigroups 
\cite{Fuzz}. In \cite{7}, Murali described the notions of \textit{%
belongingness} of a fuzzy point to a fuzzy subset under a natural
equivalence on a fuzzy subset. In \cite{8}, the idea of \textit{%
quasi-coincidence} of a fuzzy point with a fuzzy set is defined. These two
ideas acted important role in generating some different types of fuzzy
algebras. Bhakat and Das \cite{9,10} gave the concept of $\left( \alpha
,\beta \right) $-fuzzy subgroups, where $\alpha ,\beta \in \{\in ,q,\in \vee
q,\in \wedge q\}$ and $\alpha \neq \in \wedge q$. These fuzzy subgroups are
further studied in \cite{11,12}. The concept of $\left( \in ,\in \vee
q\right) $-fuzzy subgroups is a viable generalization of Rosenfeld's fuzzy
subgroups. The concept of $\left( \in ,\in \vee q\right) $-fuzzy subrings
and ideals are described. In \cite{13}, S.K. Bhakat and P. Das introduced
the $\left( \in ,\in \vee q\right) $-fuzzy subrings and ideals. Davvaz
described the notions of $\left( \in ,\in \vee q\right) $-fuzzy subnearrings
and $\left( \in ,\in \vee q\right) $-fuzzy ideals of a near ring in \cite{14}%
. Jun and Song introduced the study of $\left( \alpha ,\beta \right) $-fuzzy
interior ideals of a semigroup in \cite{15}. In \cite{16}, the concept of $%
\left( \in ,\in \vee q\right) $-fuzzy bi-ideals of a semigroup described by
Kazanci and Yamak. In \cite{17}, Jun et al gave some different type of
characterization of regular semigroups by the properties of $\left( \in ,\in
\vee q\right) $-fuzzy ideals. Aslam et al introduced a new generalization of
fuzzy $\Gamma $-ideals in $\Gamma $-LA-semigroups \cite{222}. In \cite{221},
Abdullah et al studied a new type of fuzzy normal subgroup and fuzzy coset
of groups. Generalizing the idea of the quasi-coincident of a fuzzy point
with a fuzzy subset, Jun \cite{19} defined $\left( \in ,\in \vee
q_{k}\right) $-fuzzy subalgebras in BCK/BCI-algebras. In \cite{26}, $\left(
\in ,\in \vee q_{k}\right) $-fuzzy ideals of semigroups are introduced.
Further generalizing the concept of $\left( \in ,\in \vee q\right) $, J.
Zhan and Y. Yin defined $\left( \in _{\gamma },\in _{\gamma }\vee q_{\delta
}\right) $-fuzzy ideals of near rings \cite{25}. In \cite{24}, $\left( \in
_{\gamma },\in _{\gamma }\vee q_{\delta }\right) $-fuzzy ideals of
BCI-algebras are introduced. The Abdullah et al, the concept of $\left( \in
_{\gamma },\in _{\gamma }\vee q_{\delta }\right) $-fuzzy LA-subsemigroups, $%
\left( \in _{\gamma },\in _{\gamma }\vee q_{\delta }\right) $-fuzzy
left(right) ideals, $\left( \in _{\gamma },\in _{\gamma }\vee q_{\delta
}\right) $-fuzzy generalized bi-ideals and $\left( \in _{\gamma },\in
_{\gamma }\vee q_{\delta }\right) $-fuzzy bi-ideals of an LA-semigroup are
introduced \cite{A222}.

The paper is organized as: In section 2, we give basic definition of an
LA-semigroup and a fuzzy set. In section 3, we define $\left( \in _{\gamma
},\in _{\gamma }\vee q_{\delta }\right) $-fuzzy LA-subsemigroups, $\left(
\in _{\gamma },\in _{\gamma }\vee q_{\delta }\right) $-fuzzy left(right)
ideals, $\left( \in _{\gamma },\in _{\gamma }\vee q_{\delta }\right) $-fuzzy
generalized bi-ideals and $\left( \in _{\gamma },\in _{\gamma }\vee
q_{\delta }\right) $-fuzzy bi-ideals of LA-semigroups. In section 4 we give
some characterization results of these ideals in fuzzy varsion. We also give
characterizations of LA-semigroups by the properties of $\left( \in _{\gamma
},\in _{\gamma }\vee q_{\delta }\right) $-fuzzy ideals.

\section{Preliminaries}

An LA-subsemigroup of $S$ \ means a non-empty subset $A$ of $S$ such that $%
A^{2}\subseteq A.$ By a left (right) ideals of $S$ we mean a non-empty
subset $I$ of $S$ such that $SI\subseteq I(IS\subseteq I).$ An ideal $I$ is
said to be two sided or simply ideal if it is both left and right ideal. An
LA-subsemigroup $A$ is called bi-ideal if $\left( BS\right) B\subseteq A.$ A
non-empty subset $B$ is called generalized bi-ideal if $\left( BS\right)
B\subseteq A.$ A non-empty subset $Q$ is called a quasi-ideal if $QS\cap
SQ\subseteq Q.$ A non-empty subset $A$ is called interior ideal if it is
LA-subsemigroup of $S$ and $\left( SA\right) S\subseteq A.$ An LA-semigroup $%
S$ is called regular if for each $a\in S$ there exists $x\in S$ such that $%
a=(ax)a.$ An LA-semigroup $S$ is called intra-regular if for each $a\in S$
there exist $x,y\in S$ such that $a=(xa^{2})y.$ An LA-semigroup $S$ is
called weakly regular if for each $s\in S,$ there exists $x,y\in S,$ such
that $s=\left( sx\right) (sy).$ In an LA-semigroup $S,$ the following law
hold, (1) $\left( ab\right) c=\left( ab\right) c,$ for all $a,b,c\in S.$ $%
\left( 2\right) $ $\left( ab\right) \left( cd\right) =\left( ac\right)
\left( bd\right) ,$ for all $a,b,c,d\in S.$ If an LA-semigroup $S$ has a
left identity $e,$ then the following law holds, (3) $\left( ab\right)
\left( cd\right) =\left( db\right) \left( ca\right) ,$ for all $a,b,c,d\in
S. $ (4) $a(bc)=b(ac),$ for all $a,b,c\in S.$ A fuzzy subset $\mu $ of $S$
is a mapping $\mu :S\rightarrow \lbrack 0,1].$ For any two subsets $\mu $
and $\nu $ of $S,$ the product $\mu \circ \nu $ is defined as%
\begin{equation*}
\left( \mu \circ \nu \right) (x)=\QDATOPD\{ . {\bigvee\limits_{x=yz}\{\mu
(y)\wedge \nu (z)\text{, if there exist }y,z\in S,\text{ such that }x=yz}{0%
\text{ \ \ \ \ \ \ \ \ \ \ \ \ \ \ \ \ \ \ \ \ \ \ \ \ \ \ \ \ \ \ \ \ \ \ \
\ \ \ \ \ \ \ \ \ \ \ \ \ \ \ \ \ \ \ \ \ \ \ \ otherwise \ \ \ \ \ \ \ \ \
\ }}
\end{equation*}%
A fuzzy subset $\mu $ of the form%
\begin{equation*}
\mu (y)=\left\{ 
\begin{array}{c}
t(\neq 0)\text{ if }y=x \\ 
0\text{ \ \ \ \ \ \ \ if }y\neq x\text{\ }%
\end{array}%
\right.
\end{equation*}%
is said to be a fuzzy point with support $x$ and value $t$ and is denoted by 
$x_{t}$.A fuzzy point $x_{t}$ is said to be "belong
to"(res.,"quasicoincident with") a fuzzy set $\mu ,$ written as $x_{t}\in
\mu ($repectively$,x_{t}q\mu )$ if $\mu (x)\geq t$ $($repectively, $\mu
(x)+t>1).$ We write $x_{t}\in \vee q\mu $ if $x_{t}\in \mu $ or $x_{t}q\mu .$%
If $\mu (x)<t($respectively, $\mu (x)+t\leq 1),$ then we write $x_{t}%
\overline{\in }\mu ($repectively,$x_{t}\overline{q}\mu ).$ We note that $%
\overline{\in \vee q}$ means that $\in \vee q$ does not hold. Generalizing
the concept of $x_{t}q\mu ,$ Y. B. Jun \cite{19} defined $x_{t}q_{k}\mu ,$
where $k\in \lbrack 0,1]$ as $x_{t}q_{k}\mu $ if $\mu (x)+t+\dot{k}>1$ and $%
x_{t}\in \vee q_{k}\mu $ if $x_{t}\in \mu $ or $x_{t}q_{k}\mu .$

\begin{definition}
\cite{Fuzz} A fuzzy subset $\mu $ of an LA-semigroup $S$ is called fuzzy
LA-subsemigroup $S$ if $\mu (xy)\geq \mu (x)\wedge \mu (y)$ for all $x,y\in
S.$
\end{definition}

\begin{definition}
\cite{Fuzz} A fuzzy subset $\mu $ of an LA-semigroup $S$ is called fuzzy
left(right) ideal of $S$ if $\mu (xy)\geq \mu (y)(\mu (xy)\geq \mu (x))$ for
all $x,y\in S.$
\end{definition}

\begin{definition}
\cite{Fuzz} An LA-subsemigroup $\mu $ of an LA-semigroup $S$ is called fuzzy
bi-ideal of $S$ if $\mu (\left( xy\right) z)\geq \mu (x)\wedge \mu (z)$ for
all $x,y\in S.$
\end{definition}

\begin{definition}
\cite{Fuzz} A fuzzy subset $\mu $ of an LA-semigroup $S$ is called fuzzy
generalized bi-ideal of $S$ if $\mu (\left( xy\right) z)\geq \mu (x)\wedge
\mu (z)$ for all $x,y\in S.$
\end{definition}

\begin{definition}
\cite{Fuzz}Let $\mu $ be a fuzzy subset of an LA-semigroup $S.$ Then, for
all $t\in (0,1],$ the set $\mu _{t}=\{x\in S\left\vert \mu (x)\geq t\right.
\}$ is called a level subset of $S.$
\end{definition}

Following theorems are well known in LA-semigroups

\begin{theorem}
\label{TH1}Let $S$ be a weakly regular LA-semigroup with left identity $e.$
Then, the following conditions are equivalent:

(i) $S$ is regular.

(ii) $R\cap L=LR$ for every left ideal $L$ and right ideal $R.$
\end{theorem}

\begin{theorem}
\label{TH12}Let $S$ be a weakly regular LA-semigroup with left identity $e.$
Then, the following conditions are equivalent:

(i) $S$ is regular.

(ii) $L\cap R\subseteq LR$ for every left ideal $L$ and right ideal $R.$
\end{theorem}

\begin{theorem}
\label{TH2}Let $S$ be a weakly regular LA-semigroup with left identity $e.$
Then, the following conditions are equivalent:

(i) $S$ is regular and intra-regular.

(ii) Every quasi ideal of $S$ is idempotent.
\end{theorem}

\begin{theorem}
\label{TH3}Let $S$ be a weakly regular LA-semigroup with left identity $e.$
Then, the following conditions are equivalent:

(i) $S$ is regular.

(ii) $A\cap B=BA$ for every left ideal $A$ and right ideal $B$ of $S.$

(ii) $\left( AS\right) A=A$ for every quasi ideal $A$ of $S.$
\end{theorem}

\section{$\left( \in _{\protect\gamma },\text{ }\in _{\protect\gamma }\vee
q_{\protect\delta }\right) $-FUZZY IDEALS}

Generalizing the notion of $\left( \in ,\in \vee q\right) ,$ in \cite{24,25} 
$\left( \in _{\gamma },\in _{\gamma }\vee q_{\delta }\right) $-fuzzy ideals
of near rings and BCI-algebras are introduced. Let $\gamma ,\delta \in
\lbrack 0,1]$ be such that $\gamma <\delta .$ For fuzzy point $x_{t}$ and
fuzzy subset $\mu $ of $S,$ we say

$(i)$\ $x_{t}\in _{\gamma }\mu $ if $\mu (x)\geq t>\gamma .$

$(ii)\ x_{t}q_{\delta }\mu $ \ if $\mu (x)+t>2\delta .$

$(iii)\ x_{t}\in _{\gamma }\vee q_{\delta }\mu $ if $x_{t}\in _{\gamma }\mu $
or $x_{t}q_{\delta }\mu .$

$(iv)\ x_{t}\overline{\in }_{\gamma }\vee \overline{q}_{\delta }\mu $ if $%
x_{t}\overline{\in }_{\gamma }\mu $ or $x_{t}\overline{q}_{\delta }\mu $.

In this section, we introduce the concept of $\left( \in _{\gamma },\in
_{\gamma }\vee q_{\delta }\right) $-fuzzy LA-subsemigroup, $\left( \in
_{\gamma },\in _{\gamma }\vee q_{\delta }\right) $-fuzzy left(right) ideal, $%
\left( \in _{\gamma },\in _{\gamma }\vee q_{\delta }\right) $-fuzzy
generalized bi-ideal and $\left( \in _{\gamma },\in _{\gamma }\vee q_{\delta
}\right) $-fuzzy bi-ideal of an LA-semigroup $S$. We also study some basic
properties of these ideals.

\begin{definition}
\cite{A222}A fuzzy subset $\mu $ of an LA-semigroup $S$ is called an $\left(
\in _{\gamma },\in _{\gamma }\vee q_{\delta }\right) $-fuzzy LA-subsemigroup
of $S$ if for all $a,b\in S$ and $t,r\in (\gamma ,1]$, $\ x_{t},y_{r}\in
_{\gamma }\mu $ implies that $\left( ab\right) _{t\wedge r}\in _{\gamma
}\vee q_{\delta }\mu .$
\end{definition}

\begin{remark}
\cite{A222}Every fuzzy LA-subsemigroup and every $\left( \in ,\in \vee
q\right) $-fuzzy LA-subsemigroup is an $\left( \in _{\gamma },\in _{\gamma
}\vee q_{\delta }\right) $-fuzzy LA-subsemigroup but the converse is not
true.
\end{remark}

\begin{definition}
\cite{A222}A fuzzy subset $\mu $ of an LA-semigroup $S$ is called \ an \ $%
\left( \in _{\gamma },\in _{\gamma }\vee q_{\delta }\right) $-fuzzy
generalized bi- ideal if for all $a,b,s\in S$ and $t,r\in (\gamma ,1],$ $%
a_{t}\in _{\gamma }\mu ,b_{r}\in _{r}\mu $ implies that $\left( \left(
as\right) b\right) _{t\wedge r}\in _{\gamma }\vee q_{\delta }\mu .$
\end{definition}

\begin{definition}
\cite{A222}A fuzzy subset $\mu $ of an LA-semigroup $S$ is called \ an \ $%
\left( \in _{\gamma },\in _{\gamma }\vee q_{\delta }\right) $ fuzzy bi-
ideal if for all $a,b,s\in S$ and $t,r\in (\gamma ,1]$ $\ a_{t}\in _{\gamma
}\mu ,$ $b_{r}\in _{r}\mu $ implies that

(i) $\left( ab\right) _{t\wedge r}\in _{\gamma }\vee q_{\delta }\mu .$

(ii) $\left( \left( as\right) b\right) _{t\wedge r}\in _{\gamma }\vee
q_{\delta }\mu .$
\end{definition}

\begin{definition}
\cite{A222}A fuzzy subset $\mu $ of an LA-semigroup $c$ is called an $\left(
\in _{\gamma },\in _{\gamma }\vee q_{\delta }\right) $-fuzzy interior ideal
of $c$ if for all $a,b,c\in S$ and $t,r\in (\gamma ,1],$ the following
conditions hold:

$(i)$ $a_{t},b_{r}\in _{\gamma }\mu $ implies that $\left( ab\right)
_{t\wedge r}\in _{\gamma }\vee q_{\delta }\mu .$

$(ii)$ $c_{t}\in _{\gamma }\mu $ implies that $\left( \left( ac\right)
b\right) _{t}\in _{\gamma }\vee q_{\delta }\mu .$
\end{definition}

\begin{definition}
\cite{A222}A fuzzy subset $\mu $ of an LA-semigroup $S$ is called an $\left(
\in _{\gamma },\in _{\gamma }\vee q_{\delta }\right) $-fuzzy quasi-ideal of $%
S,$\ if it satisfies,

$\mu (x)\vee \gamma \geq \left( \mu \circ 1\right) (x)\wedge \left( 1\circ
\mu \right) (x)\wedge \delta ,$
\end{definition}

\section{Regular and Weakly Regular\ LA-Semigroups.}

In this section we characterize weakly regular, regular and intra-regular
LA-semigroups by the prperties of their $\left( \in _{\gamma },\in _{\gamma
}\vee q_{\delta }\right) $-fuzzy ideals.

\begin{theorem}
Every $\left( \in _{\gamma },\in _{\gamma }\vee q_{\delta }\right) $-fuzzy
generalized bi-ideal of a weakly regular LA-semigroup $S$ with left identity 
$e$ is an $\left( \in _{\gamma },\in _{\gamma }\vee q_{\delta }\right) $%
-fuzzy bi-ideal of $S.$
\end{theorem}

\begin{proof}
Suppose that $\mu $ is an $\left( \in _{\gamma },\in _{\gamma }\vee
q_{\delta }\right) $-fuzzy generalized bi-ideal of $S.$ Let $a,b\in S.$
Since $S$ is weakly regular so there exists $x,y\in S,$ such that, $%
a=(ax)(ay).$ Then, by using $\left( 4\right) ,$ we have $ab=(a(ax)y)b).$ So, 
\begin{equation*}
\mu (ab)\vee \gamma =\mu (a((ax)y)b)\vee \gamma \geq \mu (a)\wedge \mu
(b)\wedge \delta .
\end{equation*}%
This shows that $\mu $\ is an $\left( \in _{\gamma },\in _{\gamma }\vee
q_{\delta }\right) $-fuzzy\ LA-subsemigroup of $c.$ Hence, $\mu $ is an $%
\left( \in _{\gamma },\in _{\gamma }\vee q_{\delta }\right) $-fuzzy bi-ideal
of $c.$
\end{proof}

\begin{theorem}
Every $\left( \in _{\gamma },\in _{\gamma }\vee q_{\delta }\right) $-fuzzy
quasi-ideal of a wekly regular LA-semigroup $S$ with left identity $e,$ is
an $\left( \in _{\gamma },\in _{\gamma }\vee q_{\delta }\right) $-fuzzy
bi-ideal of $S.$
\end{theorem}

\begin{proof}
Suppose that $\mu $ is an $\left( \in _{\gamma },\in _{\gamma }\vee
q_{\delta }\right) $-fuzzy quasi-ideal of $S$ and $a\in S.$ Since $S$ is
weakly regular, so their exist $x,y\in S,$ such that, $a=(ax)(ay).$ Now,%
\begin{eqnarray*}
\mu (ab)\vee \gamma &\geq &\left( \mu \circ 1\right) (ab)\wedge \left(
1\circ \mu \right) (ab)\wedge \delta \\
&=&\left[ \bigvee\limits_{ab=lm}\{\mu (l)\wedge 1(m)\}\right] \wedge \left[
\bigvee\limits_{ab=uv}\{1(u)\wedge \mu (v)\}\right] \wedge \delta \\
&\geq &\left[ \mu (a)\wedge 1(b)\right] \wedge \left[ 1(a)\wedge \mu (b)%
\right] \wedge \delta \\
&=&\left[ \mu (a)\wedge 1\right] \wedge \left[ 1\wedge \mu (b)\right] \wedge
\delta \\
&=&\mu (a)\wedge \mu (b)\wedge \delta .
\end{eqnarray*}%
Thus, $\mu $\ is an $\left( \in _{\gamma },\in _{\gamma }\vee q_{\delta
}\right) $-fuzzy LA-subsemigroup of $S.$ Now by using (4),(1),(2) and (3) we
have%
\begin{eqnarray*}
(as)b &=&(((ax)(ay))s)(eb) \\
&=&((a((ax)y))s)(eb) \\
&=&((s((ax)y))a)(ec) \\
&=&(ce)(a(s((ax)y))) \\
&=&a((ce)(s((ax)y))).
\end{eqnarray*}%
So,%
\begin{eqnarray*}
\mu (\left( as\right) b)\vee \gamma &\geq &\left( \mu \circ 1\right) (\left(
as\right) b)\wedge \left( 1\circ \mu \right) (\left( as\right) b)\wedge
\delta \\
&=&\left[ \bigvee\limits_{\left( as\right) b=a((ce)(s((ax)y)))=lm}\{\mu
(l)\wedge 1(m)\}\right] \wedge \left[ \bigvee\limits_{\left( as\right)
b=uv}\{1(u)\wedge \mu (v)\}\right] \wedge \delta \\
&\geq &\left[ \mu (a)\wedge 1((ce)(s((ax)y)))\right] \wedge \left[
1(as)\wedge \mu (b)\right] \wedge \delta \\
&=&\left[ \mu (a)\wedge 1\right] \wedge \left[ 1\wedge \mu (b)\right] \wedge
\delta \\
&=&\mu (a)\wedge \mu (b)\wedge \delta .
\end{eqnarray*}%
Thus, $\mu $ is an $\left( \in _{\gamma },\in _{\gamma }\vee q_{\delta
}\right) $-fuzzy bi-ideal of $S.$
\end{proof}

\begin{theorem}
If $\mu $ is an $\left( \in _{\gamma },\in _{\gamma }\vee q_{\delta }\right) 
$-fuzzy left ideal and $\nu $ be the $\left( \in _{\gamma },\in _{\gamma
}\vee q_{\delta }\right) $-fuzzy right ideal of a weakly regular
LA-semigroup $S$ with left identity $e,$ then $\mu \circ \nu $ is an $\left(
\in _{\gamma },\in _{\gamma }\vee q_{\delta }\right) $-fuzzy two-sided ideal
of $S.$
\end{theorem}

\begin{proof}
Let $x,y\in S,$ then by definition of weakly regular LA-semigroup there
exist $t,r\in S,$ such that, $x=\left( xr\right) (xt).$ Now,%
\begin{eqnarray*}
\left( \mu \circ \nu \right) (y)\wedge \delta &=&\left(
\bigvee\limits_{y=lm}\{\mu (l)\wedge \nu (m)\}\right) \wedge \delta \\
&=&\left( \bigvee\limits_{y=lm}\mu (l)\wedge \nu (m)\wedge \delta \right) \\
&=&\left( \bigvee\limits_{y=lm}\mu (l)\wedge \delta \wedge \nu (m)\right) .
\end{eqnarray*}%
Now, by using (1),(2),(3) and (4), we have%
\begin{eqnarray*}
xy &=&x(lm)=((xr)(xt))(lm)=((lm)(xt))(xr) \\
&=&((tx)(ml))(xr)=(m((tx)l))(xr) \\
&=&((xr)((tx)l))m=((xr)((tx))(el)))m \\
&=&((xr)((le)(xt)))m)=((xr)(x((le)t)))m \\
&=&(x((xr)((le)t)))m=(x((xr)((te)l)))m.
\end{eqnarray*}%
Now,%
\begin{eqnarray*}
\mu (x((xr)((te)l)))\vee \gamma &=&\left( \left( \mu (x((xr)((te)l)))\vee
\gamma \right) \vee \gamma \right) \vee \gamma \\
&\geq &\left( \left( \mu ((xr)((te)l))\wedge \delta \right) \vee \gamma
\right) \vee \gamma \text{ } \\
&&\text{(because }\mu \text{\ is an \ }\left( \in _{\gamma },\in _{\gamma
}\vee q_{\delta }\right) \text{-fuzzy\ left ideal) } \\
\text{ \ \ \ \ \ \ \ \ \ \ \ \ \ \ \ \ \ \ \ } &=&\left( \left( \mu
((xr)((te)l))\vee \gamma \right) \wedge \delta \right) \vee \gamma \text{\ \ 
} \\
&\geq &\text{ }\left( \left( \mu ((te)l))\wedge \delta \right) \wedge \delta
\right) \vee \gamma \\
&=&\text{ }\left( \mu ((te)l))\wedge \delta \right) \vee \gamma \\
&=&\text{ }\left( \mu ((te)l))\vee \gamma \right) \wedge \delta \\
&\geq &\mu (l))\wedge \delta \wedge \delta \\
&=&\mu (l))\wedge \delta .
\end{eqnarray*}%
Thus,%
\begin{eqnarray*}
\left( \mu \circ \nu \right) (y)\wedge \delta &=&\left(
\bigvee\limits_{y=lm}\mu (l)\wedge \delta \wedge \nu (m)\right) \\
&\leq &\bigvee\limits_{xy=(x((xr)((te)l)))m}\left( \left\{ \mu
((x((xr)((te)l)))m)\vee \gamma \right\} \wedge \nu (m)\right) \\
&\leq &\bigvee\limits_{xy=pq}\left( \left\{ \mu (p)\vee \gamma \right\}
\wedge \nu (q)\right) \\
&=&\left( \bigvee\limits_{xy=pq}\left( \mu (p)\wedge \nu (q)\right) \right)
\vee \gamma \\
&=&\left( \mu \circ \nu \right) (xy)\vee \gamma .
\end{eqnarray*}%
Thus, $\left( \mu \circ \nu \right) (xy)\vee \gamma \geq \left( \mu \circ
\nu \right) (y)\wedge \delta .$ Similarly, we can show that $\left( \mu
\circ \nu \right) (xy)\vee \gamma \geq \left( \mu \circ \nu \right)
(x)\wedge \delta .$ Hence, $\mu \circ \nu $ is an $\left( \in _{\gamma },\in
_{\gamma }\vee q_{\delta }\right) $-fuzzy two sided-ideal of $S.$
\end{proof}

Following example tells us that if $\mu $ and $\nu $ are $\left( \in
_{\gamma },\in _{\gamma }\vee q_{\delta }\right) $-fuzzy ideals of an
LA-semigroup $S,$ then $\mu \circ \nu \npreceq \mu \wedge \nu .$

\begin{example}
Let $S=\{1,2,3,4\}$ be an LA-semigroup with the following multiplication
table.%
\begin{equation*}
\begin{tabular}{l|llll}
${\ast }$ & $1$ & $2$ & $3$ & $4$ \\ \hline
$1$ & $4$ & $4$ & $4$ & $4$ \\ 
$2$ & $3$ & $1$ & $3$ & $1$ \\ 
$3$ & $4$ & $1$ & $4$ & $4$ \\ 
$4$ & $4$ & $4$ & $4$ & $4$%
\end{tabular}%
\end{equation*}

Define fuzzy subsets $\mu $ and $\nu $ of $S$ by

$\mu (1)=0.3,$ $\mu (2)=0.2,$ $\mu (3)=0.6,$ $\mu (4)=0.3.$

and

$\nu (1)=0.4,$ $\nu (2)=0.3,$ $\nu (3)=0.4,$ $\nu (4)=0.5.$

Now,

$U(\mu ;t)=\left\{ 
\begin{array}{c}
S\text{ \ \ \ \ \ \ \ \ \ \ \ \ \ \ \ \ \ \ \ \ \ \ \ \ \ \ \ \ }0<t\leq 0.2
\\ 
\{1,3,4\}\text{ \ \ \ \ \ \ \ \ \ \ \ \ \ \ \ \ \ }0.2<t\leq 0.3 \\ 
\{3\}\text{\ \ \ \ \ \ \ \ \ \ \ \ \ \ \ \ \ \ \ \ \ \ \ \ \ }0.3<t\leq 0.6
\\ 
\Phi \text{ \ \ \ \ \ \ \ \ \ \ \ \ \ \ \ \ \ \ \ \ \ \ \ \ \ \ \ \ \ \ \ \
\ \ \ }0.6<t%
\end{array}%
\right. $

and

$U(\nu ;t)=\left\{ 
\begin{array}{c}
S\text{ \ \ \ \ \ \ \ \ \ \ \ \ \ \ \ \ \ \ \ \ \ \ \ \ \ \ \ \ \ \ }0<t\leq
0.3 \\ 
\{1,3,4\}\text{ \ \ \ \ \ \ \ \ \ \ \ \ \ \ \ \ }0.3<t\leq 0.4 \\ 
\{4\}\text{\ \ \ \ \ \ \ \ \ \ \ \ \ \ \ \ \ \ \ \ \ \ \ \ }0.4<t\leq 0.5 \\ 
\emptyset \text{ \ \ \ \ \ \ \ \ \ \ \ \ \ \ \ \ \ \ \ \ \ \ \ \ \ \ \ \ \ \
\ \ \ \ \ }0.6<t%
\end{array}%
\right. $

Then, $\mu $ and $\nu $ are $\left( \in _{\gamma },\in _{\gamma }\vee
q_{\delta }\right) $-fuzzy ideals of $S$ (by using Theorem 8 part(iv))$.$
Now, $\left( \mu \circ \nu \right) (4)=\dbigvee \{0.2,0.3,0.4,0.5\}=0.5$ and 
$\left( \mu \wedge \nu \right) (4)=\mu (4)\wedge \nu (4)=0.3.$ Clearly, $\mu
\circ \nu \npreceq \mu \wedge \nu .$ Hence, $\mu \circ \nu \leq \mu \wedge
\nu $ is not true in general.
\end{example}

\begin{definition}
Let \ $\mu $ and $\nu $ be the fuzzy subsets of LA-semigroup $S.$ The fuzzy
subsets $\mu ^{\ast },\mu \wedge ^{\ast }\nu ,\mu \vee ^{\ast }\nu $ and $\
\mu \ast \nu $ of $S$ are defined as follows;%
\begin{eqnarray*}
\mu ^{\ast }(x) &=&\left( \mu (x)\vee \gamma \right) \wedge \delta \\
\left( \mu \wedge ^{\ast }\nu \right) (x) &=&\left( \left( \left( \mu \wedge
\nu \right) (x)\right) \vee \gamma \right) \wedge \delta \\
\left( \mu \vee ^{\ast }\nu \right) (x) &=&\left( \left( \left( \mu \vee \nu
\right) (x)\right) \vee \gamma \right) \wedge \delta \\
\left( \mu \ast \nu \right) (x) &=&\left( \left( \left( \mu \circ \nu
\right) (x)\right) \vee \gamma \right) \wedge \delta \text{ for all }x\in S.
\end{eqnarray*}
\end{definition}

\begin{lemma}
Let \ $\mu $ and $\nu $ be the fuzzy subsets of LA-semigroup $S.$Then the
following holds;

$(i)$ $\ \mu \wedge ^{\ast }\nu =\mu ^{\ast }\wedge \nu ^{\ast }.$

$(ii)$ $\mu \vee ^{\ast }\nu =\mu ^{\ast }\vee \nu ^{\ast }.$

$(iii)$ $\mu \ast \nu =\mu ^{\ast }\circ \nu ^{\ast }.$
\end{lemma}

\begin{proof}
Let $x\in S.$ Then,

(i)%
\begin{eqnarray*}
\left( \mu \wedge ^{\ast }\nu \right) (x) &=&\left( \left( \left( \mu \wedge
\nu \right) (x)\right) \vee \gamma \right) \wedge \delta \\
&=&\left( \left( \mu (x)\wedge \nu (x)\right) \vee \gamma \right) \wedge
\delta \\
&=&\left( \left( \mu (x)\vee \gamma \right) \wedge \left( \nu (x)\vee \gamma
\right) \right) \wedge \delta \\
&=&\left( \left( \mu (x)\vee \gamma \right) \wedge \delta )\wedge (\left(
\nu (x)\vee \gamma \right) \right) \wedge \delta ) \\
&=&\mu ^{\ast }(x)\wedge \nu ^{\ast }(x) \\
&=&\left( \mu ^{\ast }\wedge \nu ^{\ast }\right) (x).
\end{eqnarray*}%
Hence, $\ \mu \wedge ^{\ast }\nu =\mu ^{\ast }\wedge \nu ^{\ast }.$

(ii) 
\begin{eqnarray*}
\left( \mu \vee ^{\ast }\nu \right) (x) &=&\left( \left( \left( \mu \vee \nu
\right) (x)\right) \vee \gamma \right) \wedge \delta \\
&=&\left( \left( \mu (x)\vee \nu (x)\right) \vee \gamma \right) \wedge \delta
\\
&=&\left( \left( \mu (x)\vee \gamma \right) \vee \left( \nu (x)\vee \gamma
\right) \right) \wedge \delta \\
&=&\left( \left( \mu (x)\vee \gamma \right) \wedge \delta )\vee (\left( \nu
(x)\vee \gamma \right) \right) \wedge \delta ) \\
&=&\mu ^{\ast }(x)\vee \nu ^{\ast }(x) \\
&=&\left( \mu ^{\ast }\vee \nu ^{\ast }\right) (x).
\end{eqnarray*}%
Hence, $\ \mu \vee ^{\ast }\nu =\mu ^{\ast }\vee \nu ^{\ast }.$

(iii) 
\begin{eqnarray*}
\left( \mu \ast \nu \right) (x) &=&\left( \left( \left( \mu \circ \nu
\right) (x)\right) \vee \gamma \right) \wedge \delta \\
&=&\left( \left( \bigvee\limits_{x=yz}\{\mu (y)\wedge \nu (z)\}\right) \vee
\gamma \right) \wedge \delta \\
&=&\left( \bigvee\limits_{x=yz}\{\left( \mu (y)\vee \gamma \right) \wedge
\left( \nu (z)\vee \gamma \right) \}\right) \wedge \delta \\
&=&\bigvee\limits_{x=yz}\{\left( \left( \mu (y)\vee \gamma \right) \wedge
\delta \right) \wedge \left( \left( \nu (z)\vee \gamma \right) \wedge \delta
\right) \} \\
&=&\left( \bigvee\limits_{x=yz}\{\mu ^{\ast }(y)\wedge \nu ^{\ast
}(z)\}\right) \\
&=&\left( \mu ^{\ast }\circ \nu ^{\ast }\right) (x).
\end{eqnarray*}%
Hence,$\ \mu \ast \nu =\mu ^{\ast }\circ \nu ^{\ast }.$
\end{proof}

\begin{lemma}
Let $L$ and $R$ be any two non-empty subsets of LA-semigroup $S.$ Then the
followin hold.

$(i)$ $\chi _{L}\wedge ^{\ast }\chi _{R}=\chi _{L\cap R}^{\ast }.$

$(ii)\chi _{L}\vee ^{\ast }\chi _{R}=\chi _{L\cup R}^{\ast }.$

$(iii)$ $\chi _{L}\ast \chi _{R}=\chi _{LR}^{\ast }.$
\end{lemma}

\begin{lemma}
A non-empty subset $L$ of $S$ is a left(right) ideal of $S$ if and only if $%
\chi _{L}^{\ast }$ is an $\left( \in _{\gamma },\in _{\gamma }\vee q_{\delta
}\right) $-fuzzy left(right) ideal of $S.$
\end{lemma}

\begin{proof}
Suppose that $L$ be a left ideal of $S.$ Now,%
\begin{eqnarray*}
\chi _{L}^{\ast }(ab)\vee \gamma &=&\left( \left( \chi _{L}(ab)\vee \gamma
\right) \wedge \delta \right) \vee \gamma \\
&\geq &\left( \left( \chi _{L}(b)\wedge \delta \right) \wedge \delta \right)
\vee \gamma \\
&=&\left( \chi _{L}(b)\wedge \delta \right) \vee \gamma \\
&=&\left( \chi _{L}(b)\vee \gamma \right) \wedge \delta \\
&=&\chi _{L}^{\ast }(b)\wedge \delta .
\end{eqnarray*}%
Hence, $\chi _{L}^{\ast }$ is an $\left( \in _{\gamma },\in _{\gamma }\vee
q_{\delta }\right) $-fuzzy left(right) ideal of $S.$

Conversely, assume that $\chi _{L}^{\ast }$ is an $\left( \in _{\gamma },\in
_{\gamma }\vee q_{\delta }\right) $-fuzzy left ideal of $S.$ To show that $L$
is a left ideal of $S,$ let $b\in L.$ Then, 
\begin{equation*}
\chi _{L}^{\ast }(b)=\left( \chi _{L}(b)\vee \gamma \right) \wedge \delta
=\chi _{L}(b)\wedge \delta =\delta .
\end{equation*}%
So, $b_{\delta }\in _{\gamma }\chi _{L}^{\ast }.$ Now, $\left( ab\right)
_{\delta }\in _{\gamma }\vee q_{\delta }\chi _{L}^{\ast }$ (because $\chi
_{L}^{\ast }$ is an $\left( \in _{\gamma },\in _{\gamma }\vee q_{\delta
}\right) $-fuzzy left ideal of $S$), which implies that $\left( ab\right)
_{\delta }\in _{\gamma }\chi _{L}^{\ast }$ \ or $\ \left( ab\right) _{\delta
}q_{\delta }\chi _{L}^{\ast }.$ Hence, $\chi _{L}^{\ast }(ab)\geq \delta
>\gamma $ or $\chi _{L}^{\ast }(ab)+\delta >2\delta .$ Now, if $\chi
_{L}^{\ast }(ab)+\delta >2\delta ,$ then $\chi _{L}^{\ast }(ab)>\delta .$
Thus, $\chi _{L}^{\ast }(ab)\geq \delta .$ So, we have $\chi _{L}^{\ast
}(ab)=\delta .$ Hence, $ab\in L.$ Thus, $L$ is left ideal of $S.$
\end{proof}

\begin{lemma}
A non-empty subset $Q$ of $S$ is a quasi-ideal of $S$ if and only if $\chi
_{Q}^{\ast }$ is an $\left( \in _{\gamma },\in _{\gamma }\vee q_{\delta
}\right) $-fuzzy quasi-ideal ideal of $S.$
\end{lemma}

\begin{proposition}
Let $\mu $ be an $\left( \in _{\gamma },\in _{\gamma }\vee q_{\delta
}\right) $-fuzzy LA-subsemigroup of $S$. Then, $\mu ^{\ast }$ is a fuzzy
LA-subsemigroup of $S.$
\end{proposition}

\begin{proof}
Let $\mu $ be an $\left( \in _{\gamma },\in _{\gamma }\vee q_{\delta
}\right) $-fuzzy LA-subsemigroup of $S.$ Then, we have 
\begin{equation*}
\mu (ab)\vee \gamma \geq \mu (a)\wedge \mu (b)\wedge \delta ,\text{ for all }%
a,b\in S.
\end{equation*}%
Now,%
\begin{eqnarray*}
\mu ^{\ast }(ab) &=&\left( \mu (ab)\vee \gamma \right) \wedge \delta \\
&\geq &\left( \mu (a)\wedge \mu (b)\wedge \delta \right) \wedge \delta \\
&=&\left( \left( \mu (a)\vee \gamma \right) \wedge \left( \mu (b)\vee \gamma
\right) \wedge \delta \right) \\
&=&\left( \left( \mu (a)\vee \gamma \right) \wedge \delta \right) \wedge
(\left( \mu (b)\vee \gamma \right) \wedge \delta ) \\
&=&\mu ^{\ast }(a)\wedge \mu ^{\ast }(b).
\end{eqnarray*}%
Hence, $\mu ^{\ast }$ is a fuzzy LA-subsemigroup of $S.$
\end{proof}

\begin{theorem}
\label{TH4}The following assertions are equivalent for an LA-semigroup $S.$

(i) $S$ is regular.

(ii) $\mu \wedge ^{\ast }\nu =\mu \ast \nu $ for every $\left( \in _{\gamma
},\in _{\gamma }\vee q_{\delta }\right) $-fuzzy \ right ideal $\mu $ and
every $\left( \in _{\gamma },\in _{\gamma }\vee q_{\delta }\right) $-fuzzy \
left ideal $\nu $ of $S.$
\end{theorem}

\begin{proof}
(i)$\implies $ (ii) Let $\mu $ be an $\left( \in _{\gamma },\in _{\gamma
}\vee q_{\delta }\right) $-fuzzy \ right ideal and $\nu $ be an $\ \left(
\in _{\gamma },\in _{\gamma }\vee q_{\delta }\right) $-fuzzy \ left ideal of 
$S.$ Let $a\in S.$ Then, there exists $x\in S,$ such that, $a=\left(
ax\right) a.$ Now, 
\begin{eqnarray*}
\left( \mu \ast \nu \right) (a) &=&\left( \left( \left( \mu \nu \right)
(a)\right) \vee \gamma \right) \wedge \delta \\
&=&\left( \left( \tbigvee\limits_{a=xy}\{\mu (x)\wedge \nu (y)\}\right) \vee
\gamma \right) \wedge \delta \\
&\geq &\left( \left( \{\mu (ax)\wedge \nu (x)\}\right) \vee \gamma \right)
\wedge \delta \\
&&(\left( (\mu (ax)\vee \gamma \right) \wedge \nu (x))\vee \gamma )\wedge
\delta \\
&\geq &(\left( (\mu (x)\wedge \delta \right) \wedge \nu (x))\vee \gamma
)\wedge \delta \\
&=&\left( \left( \mu (x)\wedge \nu (x)\right) \vee \gamma \right) \wedge
\delta \\
&=&\left( \left( \mu (x)\wedge \nu (x)\right) \vee \gamma \right) \wedge
\delta \\
&=&\left( \mu \wedge ^{\ast }\nu \right) (x).
\end{eqnarray*}%
Thus, $\left( \mu \ast \nu \right) \geq \mu \wedge ^{\ast }\nu .$ Now,

\begin{eqnarray*}
\left( \mu \ast \nu \right) (a) &=&\left( \left( \left( \mu \nu \right)
(a)\right) \vee \gamma \right) \wedge \delta \\
&=&\left( \left( \tbigvee\limits_{a=xy}\{\mu (x)\wedge \nu (y)\}\right) \vee
\gamma \right) \wedge \delta \\
&=&\left( \tbigvee\limits_{a=xy}\left( \left( \mu (x)\wedge \delta \right)
\wedge \left( \nu (y)\wedge \delta \right) \right) \vee \gamma \right)
\wedge \delta \\
&\leq &\left( \tbigvee\limits_{a=xy}\left( \left( \mu (xy)\vee \gamma
\right) \wedge \left( \nu (xy)\vee \gamma \right) \right) \vee \gamma
\right) \wedge \delta \\
&=&\left( \tbigvee\limits_{a=xy}\mu (xy)\wedge \nu (xy)\vee \gamma \right)
\wedge \delta \\
&=&\left( \mu (a)\wedge \nu (a)\vee \gamma \right) \wedge \delta \\
&=&\ \left( \mu \wedge ^{\ast }\nu \right) (a).
\end{eqnarray*}%
Therefore, $\mu \ast \nu \leq \mu \wedge ^{\ast }\nu .$ Hence $\mu \wedge
^{\ast }\nu =\mu ^{\ast }\nu ^{\ast }.$

(ii)$\implies $(i) Let $R$ be the right ideal and $L$ be the left ideal of $%
S.$ Then, $\chi _{R}^{\ast }$ and $\chi _{L}^{\ast }$ are $\left( \in
_{\gamma },\in _{\gamma }\vee q_{\delta }\right) $-fuzzy\ right ideal and $%
\left( \in _{\gamma },\in _{\gamma }\vee q_{\delta }\right) $-fuzzy \ left
ideal of $S,$ respectively (by corollary 3). Now, by hypothesis, $\chi
_{RL}^{\ast }=\chi _{R}\ast \chi _{L}=\chi _{R}\wedge ^{\ast }\chi _{L}=\chi
_{R\cap L}^{\ast }$. Thus, $RL=R\cap L.$ Hence, it follows from THeorem \ref%
{TH1}, $S$ is regular.
\end{proof}

\begin{theorem}
For a weakly regular LA-semigroup $S$ with left identity $e,$ the following
conditions are equvialent:

(i) $S$ is regular.

(ii) $\left( \left( \mu \wedge ^{\ast }\nu \right) \wedge ^{\ast }\rho
\right) \leq \ \left( \left( \mu \ast \nu \right) \ast \rho \right) $ for
every $\left( \in _{\gamma },\in _{\gamma }\vee q_{\delta }\right) $-fuzzy \
right ideal $\mu ,$ every $\left( \in _{\gamma },\in _{\gamma }\vee
q_{\delta }\right) $-fuzzy \ generalized bi-ideal $\rho $ and every $\left(
\in _{\gamma },\in _{\gamma }\vee q_{\delta }\right) $-fuzzy \ left ideal $%
\nu $ of $S.$

(iii) $\left( \left( \mu \wedge ^{\ast }\nu \right) \wedge ^{\ast }\rho
\right) \leq \ \left( \left( \mu \ast \nu \right) \ast \rho \right) $ for
every $\left( \in _{\gamma },\in _{\gamma }\vee q_{\delta }\right) $-fuzzy \
right ideal $\mu ,$ every $\left( \in _{\gamma },\in _{\gamma }\vee
q_{\delta }\right) $-fuzzy \ bi-ideal $\rho $ and every $\left( \in _{\gamma
},\in _{\gamma }\vee q_{\delta }\right) $-fuzzy \ left ideal $\nu $ of $S.$

(iv) $\left( \left( \mu \wedge ^{\ast }\nu \right) \wedge ^{\ast }\rho
\right) \leq \ \left( \left( \mu \ast \nu \right) \ast \rho \right) $ for
every $\left( \in _{\gamma },\in _{\gamma }\vee q_{\delta }\right) $-fuzzy \
right ideal $\mu ,$ every $\left( \in _{\gamma },\in _{\gamma }\vee
q_{\delta }\right) $-fuzzy \ quasi-ideal $\rho $ and every $\left( \in
_{\gamma },\in _{\gamma }\vee q_{\delta }\right) $-fuzzy \ left ideal $\nu $
of $S.$
\end{theorem}

\begin{proof}
(i)$\Longrightarrow $(ii) Let $\mu ,$ $\nu $ and $\rho $ be any $\left( \in
_{\gamma },\in _{\gamma }\vee q_{\delta }\right) $-fuzzy \ right ideal $\mu
, $ $\left( \in _{\gamma },\in _{\gamma }\vee q_{\delta }\right) $-fuzzy \
generalized bi-ideal $\rho $ and $\left( \in _{\gamma },\in _{\gamma }\vee
q_{\delta }\right) $-fuzzy \ left ideal $\nu $ of $S,$ respectively. Let $%
s\in S.$ Since $S$ is regular, so there exist $x\in S,$ such that, $s=\left(
sx\right) s.$ Also, since $S$ is weakly regular, so there exist $y,z\in S,$
such that, $s=\left( sy\right) \left( sz\right) .$ Therefore, by using (1)
and (4) we have,%
\begin{eqnarray*}
s &=&\left( sx\right) s=(\left( \left( sy\right) \left( sz\right) \right) x)s
\\
&=&((x(sz))(sy))s=((s(xz)(sy))s \\
&=&((s(xz))(((sy)(sz))y))s=((s(xz))((s((sy)z))y)s \\
&=&((s(xz))((y((sy)z))s))s=((s(xz))(((sy)(yz))s)s \\
&=&((s(xz))((((yz)y)s)s))s \\
&=&((s(xz))((((yz)y)(sy(sz)))s))s \\
&=&((s(xz))((((yz)y)(s((sy)z)))s))s \\
&=&((s(xz))((s(((yz)y)((sy)z)))s))s.
\end{eqnarray*}%
Now,%
\begin{eqnarray*}
\left( \left( \mu \ast \nu \right) \ast \rho \right) (s) &=&\left( \left(
\bigvee\limits_{s=lm}\{\left( \mu \ast \nu \right) (l)\wedge \rho
(m)\}\right) \vee \gamma \right) \wedge \delta \\
&\geq &\left[ \left( \left( \mu \ast \nu \right)
((s(xz))((s(((yz)y)((sy)z)))s))\wedge \rho (s)\right) \vee \gamma \right]
\wedge \delta \\
&=&\left( \left( \left(
\bigvee\limits_{(s(xz))((s(((yz)y)((sy)z)))s)=lm}\{\mu (l)\wedge \nu
(m)\}\right) \wedge \rho (s)\right) \vee \gamma \right) \wedge \delta \\
&\geq &\left( \left( \{\mu (s(xz))\wedge \nu ((s(((yz)y)((sy)z)))s)\}\wedge
\rho (s)\right) \vee \gamma \right) \wedge \delta \\
&=&\left( \left( \{\left( \mu (s(xz))\vee \gamma \right) \wedge \left( \nu
((s(((yz)y)((sy)z)))s)\vee \gamma \right) \}\wedge \rho (s)\right) \vee
\gamma \right) \wedge \delta \\
&\geq &\left( \left( \{\left( \mu (s)\wedge \delta \right) \wedge \left( \nu
(s)\wedge \delta \right) \}\wedge \rho (s)\right) \vee \gamma \right) \wedge
\delta \\
&=&\left( \left( \{\mu (s)\wedge \nu (s)\}\wedge \rho (s)\right) \vee \gamma
\right) \wedge \delta \\
&=&\ \left( \left( \mu \wedge ^{\ast }\nu \right) \wedge ^{\ast }\rho
\right) (s).
\end{eqnarray*}%
Thus,$\ \left( \left( \mu \wedge ^{\ast }\nu \right) \wedge ^{\ast }\rho
\right) \leq \left( \left( \mu \ast \nu \right) \ast \rho \right) .$ $%
(ii)\implies (iii)\implies (iv)$ are obvious.

$(iv)\implies (i).$ Let $\mu $ be any $\left( \in _{\gamma },\in _{\gamma
}\vee q_{\delta }\right) $-fuzzy right ideal and $\nu $ be any $\left( \in
_{\gamma },\in _{\gamma }\vee q_{\delta }\right) $-fuzzy left ideal of $S,$
respectively. Since $1$ is $\left( \in _{\gamma },\in _{\gamma }\vee
q_{\delta }\right) $-fuzzy quasi-ideal of $S.$ So, we have%
\begin{eqnarray*}
\left( \mu \wedge ^{\ast }\nu \right) (s) &=&\left( \left( \mu \wedge \nu
\right) (s)\vee \gamma \right) \wedge \delta \\
&=&\left( \left( \left( \mu \wedge 1\right) \wedge \nu \right) (s)\vee
\gamma \right) \wedge \delta \\
&=&\left( \left( \mu \wedge ^{\ast }1\right) \wedge ^{\ast }\nu \right) (s)
\\
&\leq &\left( \left( \mu \ast 1\right) \ast \nu \right) (s) \\
&=&\left( \left( \bigvee\limits_{s=ab}\{\left( \mu \ast 1\right) (a)\wedge
\nu (b)\}\right) \vee \gamma \right) \wedge \delta \\
&=&\left( \left( \bigvee\limits_{s=ab}\left\{ \left(
\bigvee\limits_{a=lm}\mu (l)\wedge 1(m)\right) \wedge \nu (b)\right\}
\right) \vee \gamma \right) \wedge \delta \\
&=&\left( \left( \bigvee\limits_{s=ab}\left\{ \left(
\bigvee\limits_{a=lm}\mu (l)\wedge 1\right) \wedge \nu (b)\right\} \right)
\vee \gamma \right) \wedge \delta \\
&=&\left( \left( \bigvee\limits_{s=ab}\left\{ \left(
\bigvee\limits_{a=lm}\mu (l)\right) \wedge \nu (b)\right\} \right) \vee
\gamma \right) \wedge \delta \\
&=&\left( \left( \bigvee\limits_{s=ab}\left\{ \left(
\bigvee\limits_{a=lm}\mu (l)\right) \wedge \nu (b)\right\} \wedge \delta
\right) \vee \gamma \right) \wedge \delta \\
&=&\left( \left( \bigvee\limits_{s=ab}\left\{ \left(
\bigvee\limits_{a=lm}\left\{ \mu (l)\wedge \delta \right\} \right) \wedge
\nu (b)\right\} \right) \vee \gamma \right) \wedge \delta \\
&\leq &\left( \left( \bigvee\limits_{s=ab}\left\{ \left(
\bigvee\limits_{a=lm}\left\{ \mu (lm)\vee \gamma \right\} \right) \wedge \nu
(b)\right\} \right) \vee \gamma \right) \wedge \delta \\
&=&\left( \left( \bigvee\limits_{s=ab}\left\{ \left( \left\{ \mu (a)\vee
\gamma \right\} \right) \wedge \nu (b)\right\} \right) \vee \gamma \right)
\wedge \delta \\
&=&\left( \left( \bigvee\limits_{s=ab}\left\{ \mu (a)\wedge \nu (b)\right\}
\right) \vee \gamma \right) \wedge \delta \\
&=&\left( \mu \ast \nu \right) (s).
\end{eqnarray*}

Similarly, we can also prove that $\left( \mu \ast \nu \right) \leq \left(
\mu \wedge ^{\ast }\nu \right) .$ Hence, $\left( \mu \ast \nu \right)
=\left( \mu \wedge ^{\ast }\nu \right) $ for every $\left( \in _{\gamma
},\in _{\gamma }\vee q_{\delta }\right) $-fuzzy right ideal $\mu $\ and
every $\left( \in _{\gamma },\in _{\gamma }\vee q_{\delta }\right) $-fuzzy
left ideal $\nu $ of $S.$ Hence, by Theorem \ref{TH4} $S$ is regular.
\end{proof}

\begin{theorem}
\label{TH5}For a weakly regular LA-semigroup $S$ with left identity $e,$ the
following conditions are equvialent:

(i) $S$ is regular.

(ii) $\mu ^{\ast }=$ $\left( \left( \mu \ast 1\right) \ast \mu \right) $ for
every $\left( \in _{\gamma },\in _{\gamma }\vee q_{\delta }\right) $-fuzzy
generalized bi-ideal $\mu $ of $S.$

(iii) $\mu ^{\ast }=$ $\left( \left( \mu \ast 1\right) \ast \mu \right) $
for every $\left( \in _{\gamma },\in _{\gamma }\vee q_{\delta }\right) $%
-fuzzy bi-ideal $\mu $ of $S.$

(iv) $\mu ^{\ast }=$ $\left( \left( \mu \ast 1\right) \ast \mu \right) $ for
every $\left( \in _{\gamma },\in _{\gamma }\vee q_{\delta }\right) $-fuzzy
quasi-ideal $\mu $ of $S.$
\end{theorem}

\begin{proof}
(i)$\implies $(ii) Let $\mu $ be an $\left( \in _{\gamma },\in _{\gamma
}\vee q_{\delta }\right) $-fuzzy generalized bi-ideal of $S$ and $a\in S.$
Since $S$ is regular, so there exist $x\in S,$ such that, $a=\left(
ax\right) a.$ Thus, we have, 
\begin{eqnarray*}
\left( \left( \mu \ast 1\right) \ast \mu \right) (a) &=&\left( \left(
\tbigvee\limits_{a=yz}\left( \mu \ast 1\right) (y)\wedge \mu (z)\right) \vee
\gamma \right) \wedge \delta \\
&\geq &\left( \left( \left( \mu \ast 1\right) (ax)\wedge \mu (a)\right) \vee
\gamma \right) \wedge \delta \\
&=&\left( \left( \left[ \left( \left( \tbigvee\limits_{ax=bc}(\mu (b)\wedge
1(c)\right) \vee \gamma \right) \wedge \delta \right] \wedge \mu (a)\right)
\vee \gamma \right) \wedge \delta \\
&=&(([((\mu (a)\wedge 1(x))\vee \gamma )\wedge \delta ]\wedge \mu (a))\vee
\gamma )\wedge \delta \\
&=&((\mu (a)\wedge \mu (a))\vee \gamma )\wedge \delta \\
&=&(\mu (a)\vee \gamma )\wedge \delta \\
&=&\mu ^{\ast }(a).
\end{eqnarray*}%
Thus, $\left( \left( \mu \ast 1\right) \ast \mu \right) \geq \mu ^{\ast }.$
Since $\mu $ is an $\left( \in _{\gamma },\in _{\gamma }\vee q_{\delta
}\right) $-fuzzy generalized bi-ideal of $S.$ So, we have%
\begin{eqnarray*}
\left( \left( \mu \ast 1\right) \ast \mu \right) (a) &=&\left( \left(
\tbigvee\limits_{a=yz}\left( \mu \ast 1\right) (y)\wedge \mu (z)\right) \vee
\gamma \right) \wedge \delta \\
&=&\left( \left( \tbigvee\limits_{a=yz}\left( \left[ \left( \left(
\tbigvee\limits_{y=pq}\left( \mu (p)\wedge 1(q)\right) \right) \vee \gamma
\right) \wedge \delta \right] \wedge \mu (z)\right) \right) \vee \gamma
\right) \wedge \delta \\
&=&\left( \left( \tbigvee\limits_{a=yz}\left( \left(
\tbigvee\limits_{y=pq}\left( \mu (p)\wedge 1\right) \right) \wedge \mu
(z)\right) \right) \vee \gamma \right) \wedge \delta \\
&=&\left( \left( \tbigvee\limits_{a=yz}\left( \tbigvee\limits_{y=pq}\mu
(p)\wedge \mu (z)\right) \right) \vee \gamma \right) \wedge \delta \\
&=&\left( \left( \tbigvee\limits_{a=yz}\left( \tbigvee\limits_{y=pq}\left(
\mu (p)\wedge \mu (z)\right) \wedge \delta \right) \right) \vee \gamma
\right) \wedge \delta \\
&\leq &\left( \left( \tbigvee\limits_{a=\left( pq\right) z}\left( \mu
(\left( pq\right) z\right) \vee \gamma \right) \vee \gamma \right) \wedge
\delta \\
&=&\left( \left( \mu (a\right) \vee \gamma \right) \wedge \delta \\
&=&\mu ^{\ast }(a).
\end{eqnarray*}%
Thus, $\left( \left( \mu \ast 1\right) \ast \mu \right) \leq \mu ^{\ast }.$
Hence, $\left( \left( \mu \ast 1\right) \ast \mu \right) =\mu ^{\ast }.$ $%
(ii)\implies (iii)\implies (iv)$ are obvious. $(iv)\implies (i).$ Let $A$ be
any quasi-ideal of $S.$ Then, $\chi _{A}$ is an $\left( \in _{\gamma },\in
_{\gamma }\vee q_{\delta }\right) $-fuzzy quasi-ideal of $S.$ Hence, by
hypothesis, $\chi _{A}^{\ast }=\left( \left( \chi _{A}\ast 1\right) \ast
\chi _{A}\right) =\left( \left( \chi _{A}\ast \chi _{S}\right) \ast \chi
_{A}\right) =\chi _{\left( AS\right) A}^{\ast }.$ This implies that $%
A=\left( AS\right) A.$ Hence, it follows from Theorem \ref{TH3}, that $S$ is
regular.
\end{proof}

\begin{theorem}
For a weakly regular LA-semigroup $S$ with left identity $e,$ the following
conditions are equvialent:

(i) $S$ is regular.

(ii) $\left( \mu \wedge ^{\ast }\nu \right) =$ $\left( \left( \mu \ast \nu
\right) \ast \mu \right) $ for every $\left( \in _{\gamma },\in _{\gamma
}\vee q_{\delta }\right) $-fuzzy quasi-ideal $\mu $ and every $\left( \in
_{\gamma },\in _{\gamma }\vee q_{\delta }\right) $-fuzzy ideal $\nu $ of $S.$

(iii) $\left( \mu \wedge ^{\ast }\nu \right) =$ $\left( \left( \mu \ast \nu
\right) \ast \mu \right) $ for every $\left( \in _{\gamma },\in _{\gamma
}\vee q_{\delta }\right) $-fuzzy quasi-ideal $\mu $ and every $\left( \in
_{\gamma },\in _{\gamma }\vee q_{\delta }\right) $-fuzzy interior ideal $\nu 
$ of $S.$

(iv) $\left( \mu \wedge ^{\ast }\nu \right) =$ $\left( \left( \mu \ast \nu
\right) \ast \mu \right) $ for every $\left( \in _{\gamma },\in _{\gamma
}\vee q_{\delta }\right) $-fuzzy bi-ideal $\mu $ and every $\left( \in
_{\gamma },\in _{\gamma }\vee q_{\delta }\right) $-fuzzy ideal $\nu $ of $%
S.. $

(v) $\left( \mu \wedge ^{\ast }\nu \right) =$ $\left( \left( \mu \ast \nu
\right) \ast \mu \right) $ for every $\left( \in _{\gamma },\in _{\gamma
}\vee q_{\delta }\right) $-fuzzy bi-ideal $\mu $ and every $\left( \in
_{\gamma },\in _{\gamma }\vee q_{\delta }\right) $-fuzzy interior ideal $\nu 
$ of $S.$

(vi) $\left( \mu \wedge ^{\ast }\nu \right) =$ $\left( \left( \mu \ast \nu
\right) \ast \mu \right) $ for every $\left( \in _{\gamma },\in _{\gamma
}\vee q_{\delta }\right) $-fuzzy generalized bi-ideal $\mu $ and every $%
\left( \in _{\gamma },\in _{\gamma }\vee q_{\delta }\right) $-fuzzy ideal $%
\nu $ of $S.$

(vii) $\left( \mu \wedge ^{\ast }\nu \right) =$ $\left( \left( \mu \ast \nu
\right) \ast \mu \right) $ for every $\left( \in _{\gamma },\in _{\gamma
}\vee q_{\delta }\right) $-fuzzy generalized bi-ideal $\mu $ and every $%
\left( \in _{\gamma },\in _{\gamma }\vee q_{\delta }\right) $-fuzzy interior
ideal $\nu $ of $S.$
\end{theorem}

\begin{proof}
(i)$\implies $(vii) Let $\mu $ be an $\left( \in _{\gamma },\in _{\gamma
}\vee q_{\delta }\right) $-fuzzy generalized bi-ideal and $\nu $ be an $%
\left( \in _{\gamma },\in _{\gamma }\vee q_{\delta }\right) $-fuzzy interior
ideal of $S,$ and $x\in S.$ Then,%
\begin{eqnarray*}
\left( \left( \mu \ast \nu \right) \ast \mu \right) (x) &\leq &\left( \left(
\mu \ast 1\right) \ast \mu \right) (x) \\
&=&\left( \left( \tbigvee\limits_{x=yz}\left( \mu \ast 1\right) (y)\wedge
\mu (z)\right) \vee \gamma \right) \wedge \delta \\
&=&\left( \left( \tbigvee\limits_{x=yz}\left( \left[ \left( \left(
\tbigvee\limits_{y=pq}\left( \mu (p)\wedge 1(q)\right) \right) \vee \gamma
\right) \wedge \delta \right] \wedge \mu (z)\right) \right) \vee \gamma
\right) \wedge \delta \\
&=&\left( \left( \tbigvee\limits_{x=yz}\left( \left(
\tbigvee\limits_{y=pq}\left( \mu (p)\wedge 1\right) \right) \wedge \mu
(z)\right) \right) \vee \gamma \right) \wedge \delta \\
&=&\left( \left( \tbigvee\limits_{x=yz}\left( \tbigvee\limits_{y=pq}\mu
(p)\wedge \mu (z)\right) \right) \vee \gamma \right) \wedge \delta \\
&=&\left( \left( \tbigvee\limits_{x=yz}\left( \left(
\tbigvee\limits_{y=pq}\mu (p)\wedge \mu (z)\right) \wedge \delta \right)
\right) \vee \gamma \right) \wedge \delta \\
&\leq &\left( \left( \tbigvee\limits_{x=\left( pq\right) z}\mu (\left(
pq\right) z)\right) \vee \gamma \right) \wedge \delta \\
&=&\left( \mu (x)\vee \gamma \right) \wedge \delta \\
&=&\mu ^{\ast }(x).
\end{eqnarray*}%
Therefore $\left( \left( \mu \ast \nu \right) \ast \mu \right) \leq \mu
^{\ast }.$ Also%
\begin{eqnarray*}
\left( \left( \mu \ast \nu \right) \ast \mu \right) (x) &\leq &\left( \left(
1\ast \nu \right) \ast 1\right) (x) \\
&=&\left( \left( \tbigvee\limits_{x=yz}\left( \left( 1\ast \nu \right)
(y)\right) \wedge 1\left( z\right) \right) \vee \gamma \right) \wedge \delta
\\
&=&\left( \left( \tbigvee\limits_{x=yz}\left( \left(
\tbigvee\limits_{y=pq}\left( 1\wedge \nu (q)\right) \right) \wedge 1\right)
\right) \vee \gamma \right) \wedge \delta \\
&=&\left( \tbigvee\limits_{x=yz}\left( \tbigvee\limits_{y=pq}\left( \nu
(q)\right) \right) \vee \gamma \right) \wedge \delta \\
&=&\left( \tbigvee\limits_{x=yz}\left( \tbigvee\limits_{y=pq}\left( \nu
(q)\right) \wedge \delta )\right) \vee \gamma \right) \wedge \delta \\
&\leq &\left( \left( \tbigvee\limits_{x=\left( pq\right) z}\nu (\left(
pq)z\right) \right) \vee \gamma \right) \wedge \delta \\
&=&\left( \left( \nu (x\right) \vee \gamma \right) \wedge \delta \\
&=&\nu ^{\ast }(x).
\end{eqnarray*}%
Thus, $\left( \left( \mu \ast \nu \right) \ast \mu \right) \leq \left( \mu
^{\ast }\wedge \nu ^{\ast }\right) =\left( \mu \wedge ^{\ast }\nu \right) $.
Now, let $s\in S.$ Since $S$ is regular so there exist $x\in S,$ such that $%
s=\left( sx\right) s.$ Now, by using (1) and (4), we have%
\begin{equation*}
s=\left( sx\right) s=(((sx)s)x)s=((xs)(sx))s=(s((xs)x))s
\end{equation*}%
So, we have%
\begin{eqnarray*}
&&\left( \left( \mu \ast \nu \right) \ast \mu \right) (s) \\
&=&\left( \left( \tbigvee\limits_{s=yz}\left( (\mu \ast \nu )(y)\wedge \mu
(z)\right) \right) \vee \gamma \right) \wedge \delta \\
&\geq &\left( \left( (\mu \ast \nu )(s((xs)x)\wedge \mu (s)\right) \vee
\gamma \right) \wedge \delta \\
&=&\left( \left( \left( \left( \left( \tbigvee\limits_{s((xs)x=pq}(\mu
(p)\wedge \nu (q))\right) \vee \gamma \right) \wedge \delta \right) \wedge
\mu (s)\right) \vee \gamma \right) \wedge \delta \\
&\geq &\left( \left( (\mu (s)\wedge \nu (\left( xs\right) x))\wedge \mu
(s)\right) \vee \gamma \right) \wedge \delta \\
&=&\left( \left( \left( \mu (s)\wedge \left( \nu (\left( xs\right) x)\vee
\gamma \right) \right) \wedge \mu (s)\right) \vee \gamma \right) \wedge
\delta \\
\ &\geq &\left( \left( \left( \mu (s)\wedge \left( \nu (s)\wedge \delta \
\right) \right) \wedge \mu (s)\right) \vee \gamma \right) \wedge \delta \\
&=&\left( \left( \mu (s)\wedge \nu (s)\right) \vee \gamma \right) \wedge
\delta \\
&=&\ \left( \left( \mu \wedge \nu \right) (s)\vee \gamma \right) \wedge
\delta \  \\
&=&\ \left( \mu \wedge ^{\ast }\nu \right) (s).
\end{eqnarray*}%
Therefore, $\left( \left( \mu \ast 1\right) \ast \mu \right) \geq \left( \mu
\wedge ^{\ast }\nu \right) .$ Hence, $\left( \left( \mu \ast 1\right) \ast
\mu \right) =\left( \mu \wedge ^{\ast }\nu \right) .$\ $(vii)\implies
(v)\implies (iii)\implies (ii)$ and $(vii)\implies (vi)\implies (iv)\implies
(ii)$ are clear. $(ii)\implies (i).$\ Let $\mu $ be an \ $\left( \in
_{\gamma },\in _{\gamma }\vee q_{\delta }\right) $-fuzzy quasi-ideal of $S.$
Now, $\mu ^{\ast }(a)=\left( \mu (a)\vee \gamma \right) \wedge \delta
=\left( \left( \mu \wedge 1\right) (a)\vee \gamma \right) \wedge \delta
=\left( \mu \wedge ^{\ast }1\right) (a)=\left( \left( \mu \ast 1\right) \ast
\mu \right) (a)\Rightarrow \mu ^{\ast }=\left( \mu \ast 1\right) \ast \mu .$
Thus, by \ Theorem \ref{TH5}, $S$ is regular.
\end{proof}

\begin{theorem}
For a weakly regular LA-semigroup $S$ with left identity $e,$ the following
conditions are equvialent:

(i) $S$ is regular.

(ii) $\left( \mu \wedge ^{\ast }\nu \right) \leq \left( \mu \ast \nu \right) 
$ for every $\left( \in _{\gamma },\in _{\gamma }\vee q_{\delta }\right) $%
-fuzzy quasi-ideal $\mu $ and $\left( \in _{\gamma },\in _{\gamma }\vee
q_{\delta }\right) $-fuzzy left ideal $\nu $ of $S.$

(iii) $\left( \mu \wedge ^{\ast }\nu \right) \leq \left( \mu \ast \nu
\right) $ for every $\left( \in _{\gamma },\in _{\gamma }\vee q_{\delta
}\right) $-fuzzy bi-ideal $\mu $ and $\left( \in _{\gamma },\in _{\gamma
}\vee q_{\delta }\right) $-fuzzy left ideal $\nu $ of $S.$

(iv) $\left( \mu \wedge ^{\ast }\nu \right) \leq \left( \mu \ast \nu \right) 
$ for every $\left( \in _{\gamma },\in _{\gamma }\vee q_{\delta }\right) $%
-fuzzy generalized bi-ideal $\mu $ and $\left( \in _{\gamma },\in _{\gamma
}\vee q_{\delta }\right) $-fuzzy left ideal $\nu $ of $S.$
\end{theorem}

\begin{proof}
$(i)\implies (iv).$ Let $\mu $\ be any$\left( \in _{\gamma },\in _{\gamma
}\vee q_{\delta }\right) $-fuzzy generalized bi-ideal and \ $\nu $ be any $%
\left( \in _{\gamma },\in _{\gamma }\vee q_{\delta }\right) $-\ fuzzy left
ideal of $S.$ Let $s\in S$. Since $S$ is regular, so there exist $x\in S,$
such that, $s=\left( sx\right) s.$ Also since $S$ is weakly regular, so
there exist $y,z\in S,$ such that, $s=\left( sy\right) \left( sz\right) .$
Now, by using (4) and (1) we have%
\begin{eqnarray*}
s &=&\left( sy\right) \left( sz\right) =s((sy)z)=s((((sx)s)y)z) \\
&=&s(((ys)(sx))z)=s((s((ys)x))z)=s((z((ys)x))s)
\end{eqnarray*}%
So,%
\begin{eqnarray*}
\left( \mu \ast \nu \right) (s) &=&\left( \left( \mu \nu \right) (s)\vee
\gamma \right) \wedge \delta \\
\ \ \ \ &=&\left( \left( \tbigvee\limits_{s=yz}(\mu (y)\wedge \nu
(z))\right) \vee \gamma \right) \wedge \delta \\
&\geq &\left( \left( (\mu (s)\wedge \nu ((z((ys)x))s)\right) \vee \gamma
\right) \wedge \delta \\
\ &\geq &(\left( \mu (s)\wedge \nu (s))\vee \gamma \right) \wedge \delta \\
\ &=&\left( \left( (\mu \wedge \nu )(s)\right) \vee \gamma \right) \wedge
\delta \\
\ \ &=&(\mu \wedge ^{\ast }\nu )(s).
\end{eqnarray*}%
So, $\left( \mu \wedge ^{\ast }\nu \right) \leq \left( \mu \ast \nu \right)
. $ $(iv)\implies (iii)\implies (ii)$ are obvious. $(ii)\implies (i).$ Let $%
\mu $\ be any $\left( \in _{\gamma },\in _{\gamma }\vee q_{\delta }\right) $%
-fuzzy right ideal and\ $\nu $ be any $\left( \in _{\gamma },\in _{\gamma
}\vee q_{\delta }\right) $-\ fuzzy left ideal of $S.$ Since every $\left(
\in _{\gamma },\in _{\gamma }\vee q_{\delta }\right) $-fuzzy right ideal is
an $\left( \in _{\gamma },\in _{\gamma }\vee q_{\delta }\right) $-fuzzy
quasi-ideal of $S.$ So, $\left( \mu \wedge ^{\ast }\nu \right) \leq \left(
\mu \ast \nu \right) .$ Now,%
\begin{eqnarray*}
\left( \mu \ast \nu \right) (s) &=&\left( \left( \mu \nu \right) (s)\vee
\gamma \right) \wedge \delta \\
\ &=&\left( \left( \tbigvee\limits_{s=yz}(\mu (y)\wedge \nu (z))\right) \vee
\gamma \right) \wedge \delta \\
&=&\left( \left( \tbigvee\limits_{s=yz}(\mu (y)\wedge \nu (z))\wedge \delta
\right) \vee \gamma \right) \wedge \delta \\
&=&\left( \left( \tbigvee\limits_{s=yz}\left( (\mu (y)\wedge \delta \right)
\wedge \left( \nu (z))\wedge \delta \right) \right) \vee \gamma \right)
\wedge \delta \\
\ &\leq &\left( \left( \tbigvee\limits_{s=yz}\left( (\mu (yz)\vee \gamma
\right) \wedge \left( \nu (yz)\vee \gamma \right) \right) \vee \gamma
\right) \wedge \delta \\
\ \ \ \ &=&\left( (\mu (s)\wedge \nu (s))\vee \gamma \right) \wedge \delta \\
\ &=&\left( (\mu \wedge \nu )(s))\vee \gamma \right) \wedge \delta \\
\ \ &=&(\mu \wedge ^{\ast }\nu )(s).
\end{eqnarray*}%
So, $\mu \ast \nu \leq \mu \wedge ^{\ast }\nu .$Thus $\mu \ast \nu =\mu
\wedge ^{\ast }\nu $ for every $\left( \in _{\gamma },\in _{\gamma }\vee
q_{\delta }\right) $-fuzzy right ideal $\mu $ and $\left( \in _{\gamma },\in
_{\gamma }\vee q_{\delta }\right) $-fuzzy left ideal $\nu $ of $S.$ So, by
Theorem 31 we have $S$ is regular.
\end{proof}

\begin{theorem}
\label{TH6}For a weakly regular LA-semigroup $S$ with left identity $e,$ the
following conditions are equvialent:

(i) $S$ is intra-regular.

(ii) $\mu \wedge ^{\ast }\nu \leq \mu \ast \nu $ for every $\left( \in
_{\gamma },\in _{\gamma }\vee q_{\delta }\right) $-fuzzy left ideal $\mu $
and $\left( \in _{\gamma },\in _{\gamma }\vee q_{\delta }\right) $-fuzzy
right ideal$\nu $ of $S.$
\end{theorem}

\begin{proof}
$(i)\implies (ii).$ Let $\mu $\ be any $\left( \in _{\gamma },\in _{\gamma
}\vee q_{\delta }\right) $-fuzzy left ideal and\ $\nu $ be any $\left( \in
_{\gamma },\in _{\gamma }\vee q_{\delta }\right) $-fuzzy right ideal of $S.$
Let $s\in S$. Since $S$ is intra-regular, so there exist $x,y\in S,$ such
that, $s=\left( xs^{2}\right) y.$ Also, since $S$ is weakly regular, so
there exist $p,q\in S,$ such that, $s=\left( sp\right) \left( sq\right) .$
Now, by using (4),(1),(2) and (3) we have%
\begin{eqnarray*}
s &=&(xs^{2})y=(x(ss))y=(s(xs))y=(y(xs))s \\
\ &=&(y(xs))(es)=(se)((xs)y)=(xs)((se)y) \\
\ &=&(xs((((sp)(sq))e)=(xs)(((sq)(sp))y) \\
\ \ &=&(xs)((y(sp))(sq))=(xs)(s((y(sp))q))
\end{eqnarray*}%
So,%
\begin{eqnarray*}
\left( \mu \ast \nu \right) (s) &=&\left( \left( \left( \mu \nu \right)
(s)\right) \vee \gamma \right) \wedge \delta \\
\ &=&\left( \left( \tbigvee\limits_{s=yz}(\mu (y)\wedge \nu (z))\right) \vee
\gamma \right) \wedge \delta \\
\ &\geq &\left( (\mu (xs)\wedge \nu (s((y(sp))q)))\vee \gamma \right) \wedge
\delta \\
\ &=&\left( \left( (\mu (xs)\vee \gamma )\wedge \left( \nu (s((y(sp))q))\vee
\gamma \right) \right) \vee \gamma \right) \wedge \delta \\
&\geq &\left( \left( \left( \mu (s)\wedge \delta \right) \wedge \left( \nu
(s)\wedge \delta \right) \right) \vee \gamma \right) \wedge \delta \\
&=&\left( \left( \mu (s)\wedge \nu (s)\right) \vee \gamma \right) \wedge
\delta \\
\ &=&\left( \left( \mu \wedge \nu \right) (s))\vee \gamma \right) \wedge
\delta \\
&=&\left( \mu \wedge ^{\ast }\nu \right) (s)
\end{eqnarray*}%
Therefore, $\left( \mu \wedge ^{\ast }\nu \right) \leq \left( \mu \ast \nu
\right) $ for every $\left( \in _{\gamma },\in _{\gamma }\vee q_{\delta
}\right) $-fuzzy left ideal $\mu $ and $\left( \in _{\gamma },\in _{\gamma
}\vee q_{\delta }\right) $-fuzzy right ideal $\nu $ of $S.$ $(ii)\implies
(i) $. Let $R$ and $L$ be right and left ideals of \ $S,$ then $\chi _{R}$
and $\chi _{L}$ are $\left( \in _{\gamma },\in _{\gamma }\vee q_{\delta
}\right) $-fuzzy right ideal and $\left( \in _{\gamma },\in _{\gamma }\vee
q_{\delta }\right) $-fuzzy left ideals of $S,$ respectively. By hypothesis,
we have $\chi _{LR}^{\ast }=\left( \chi _{L}\ast \chi _{R}\right) \geq
\left( \chi _{L}\wedge ^{\ast }\chi _{R}\right) =\chi _{L\cap R}^{\ast }.$
Thus, $L\cap R\subseteq LR.$ Hence, it follows from Theorem 2, that $S$ is
intera-regular.
\end{proof}

\begin{theorem}
For a weakly regular LA-semigroup $S$ with left identity $e,$ the following
conditions are equivalent:

(i) $S$ is regular and intra-regular.

(ii) $\mu \ast \mu =\mu ^{\ast }$ for every $\left( \in _{\gamma },\in
_{\gamma }\vee q_{\delta }\right) $-fuzzy quasi- ideal $\mu $ of $S.$

(iii) $\mu \ast \mu =\mu ^{\ast }$ for every $\left( \in _{\gamma },\in
_{\gamma }\vee q_{\delta }\right) $-fuzzy bi- ideal $\mu $ of $S.$

(iv) $\mu \ast \nu \geq \mu \wedge ^{\ast }\nu $ for every $\left( \in
_{\gamma },\in _{\gamma }\vee q_{\delta }\right) $-fuzzy quasi- ideals $\mu $
and $\nu $ of $S.$

(iv)$\mu \ast \nu \geq \mu \wedge ^{\ast }\nu $ for every $\left( \in
_{\gamma },\in _{\gamma }\vee q_{\delta }\right) $-fuzzy quasi- ideal $\mu $
and for every $\left( \in _{\gamma },\in _{\gamma }\vee q_{\delta }\right) $%
-fuzzy bi- ideals $\nu $ of $S.$

(vi)$\mu \ast \nu \geq \mu \wedge ^{\ast }\nu $ for every $\left( \in
_{\gamma },\in _{\gamma }\vee q_{\delta }\right) $-fuzzy bi- ideals $\mu $
and $\nu $ of $S.$
\end{theorem}

\begin{proof}
$(i)\implies (iv)$. Let $\mu $ and $\nu $ be $\left( \in _{\gamma },\in
_{\gamma }\vee q_{\delta }\right) $-fuzzy bi-ideals of $S$ and $s\in S.$
Since $S$ is regular and intra-regular, so there exists $x,y$ and $z\in S,$
such that, $s=\left( sx\right) s$ and $s=\left( ys^{2}\right) z.$ Also, $S$
is weakly regular so there exist $p,q\in S,$ such that, $s=(sp)(sq).$ Now by
using (1) and (4) we have,%
\begin{eqnarray*}
s &=&(sx)s=(sx)((sx)s)=(((sp)(sq))x)((sx)s) \\
&=&((s((sp)q))x)((sx)s)=((x(sp)q))s)((sx)s) \\
&=&((x((sp)sq))p)q))s)((sx)s) \\
&=&((x((qp)((sp)(sq))))s)((sx)s) \\
&=&((x((qp)(s((sp)q))))s)((sx)s) \\
&=&((x(s((qp)((sp)q))))s)(sx)s) \\
&=&((s(x((qp)((sp)q))))s)((sx)s).
\end{eqnarray*}%
Then,%
\begin{eqnarray*}
\left( \mu \ast \nu \right) \left( s\right) &=&\left( \left( \mu \nu \right)
(s)\vee \gamma \right) \wedge \delta \\
&=&\left( \left( \bigvee_{s=yz}\{\mu (y)\wedge \nu (z)\}\right) \vee \gamma
\right) \wedge \delta \\
&\geq &\left( \{\mu ((s(x((qp)((sp)q))))s)\wedge \nu ((sx)s)\}\vee \gamma
\right) \wedge \delta \\
&=&\left[ \left( \{\left( \mu ((s(x((qp)((sp)q))))s)\vee \gamma \right)
\wedge (\nu ((sx)s)\right) \vee \gamma )\}\vee \gamma \right] \wedge \delta
\\
&\geq &\left[ \{\left( \mu (s)\wedge \delta \right) \wedge (\nu (s)\wedge
\delta )\}\vee \gamma \right] \wedge \delta \\
&=&\left[ \{\mu (s)\wedge \nu (s)\}\vee \gamma \right] \wedge \delta \\
&=&\left[ \{\left( \mu \wedge \nu \right) (s)\}\vee \gamma \right] \wedge
\delta \\
&=&\left( \mu \wedge ^{\ast }\nu \right) (s).
\end{eqnarray*}%
Thus, $\mu \ast \nu \geq \mu \wedge ^{\ast }\nu $ for every $\left( \in
_{\gamma },\in _{\gamma }\vee q_{\delta }\right) $-fuzzy bi- ideals $\mu $
and $\nu $ of $S$. $(vi)\implies (v)\implies (iv)$ are obvious. $%
(iv)\implies (ii).$ Put $\mu =\nu $ in $(iv),$ we get $\mu \ast \mu \geq \mu
^{\ast }.$ Since every $\left( \in _{\gamma },\in _{\gamma }\vee q_{\delta
}\right) $-fuzzy quasi-ideal is $\left( \in _{\gamma },\in _{\gamma }\vee
q_{\delta }\right) $-fuzzy LA-subsemigroup, so $\mu \ast \mu \leq \mu ^{\ast
}.$ Thus, $\mu \ast \mu =\mu ^{\ast }.$ $(iii)\implies (ii)$ is obvious. $%
(ii)\implies (i).$ Let $Q$ be a quasi ideal of $S.$ Then, by lemma 3$,$ $%
\chi _{Q}$ is an $\left( \in _{\gamma },\in _{\gamma }\vee q_{\delta
}\right) $-fuzzy quasi-ideal of $S.$ Hence, by hypothesis, $\chi _{Q}\ast
\chi _{Q}=\chi _{Q}^{\ast }.$ Thus, $\chi _{QQ}^{\ast }=\chi _{Q}\ast \chi
_{Q}=\chi _{Q}^{\ast },$ implies that $QQ=Q.$ So, by Theorem \ref{TH2}, $S$
is both regular and intra-regular.
\end{proof}

\begin{theorem}
For a weakly regular LA-semigroup $S$ with left identity $e,$ the following
conditions are equivalent:

(i) $S$ is regular and intra-regular.

(ii) $\left( \mu \ast \nu \right) \wedge \left( \nu \ast \mu \right) \geq
\mu \wedge ^{\ast }\nu $ for every $\left( \in _{\gamma },\in _{\gamma }\vee
q_{\delta }\right) $-fuzzy right ideal $\mu $ and every $\left( \in _{\gamma
},\in _{\gamma }\vee q_{\delta }\right) $-fuzzy left ideal $\nu $ of $S.$

(iii) $\left( \mu \ast \nu \right) \wedge \left( \nu \ast \mu \right) \geq
\mu \wedge ^{\ast }\nu $ for every $\left( \in _{\gamma },\in _{\gamma }\vee
q_{\delta }\right) $-fuzzy right ideal $\mu $ and every $\left( \in _{\gamma
},\in _{\gamma }\vee q_{\delta }\right) $-fuzzy quasi- ideal $\nu $ of $S.$

(iv) $\left( \mu \ast \nu \right) \wedge \left( \nu \ast \mu \right) \geq
\mu \wedge ^{\ast }\nu $ for every $\left( \in _{\gamma },\in _{\gamma }\vee
q_{\delta }\right) $-fuzzy right ideal $\mu $ and every $\left( \in _{\gamma
},\in _{\gamma }\vee q_{\delta }\right) $-fuzzy bi-ideal $\nu $ of $S.$

(v) $\left( \mu \ast \nu \right) \wedge \left( \nu \ast \mu \right) \geq \mu
\wedge ^{\ast }\nu $ for every $\left( \in _{\gamma },\in _{\gamma }\vee
q_{\delta }\right) $-fuzzy right ideal $\mu $ and every $\left( \in _{\gamma
},\in _{\gamma }\vee q_{\delta }\right) $-fuzzy generideal bi-ideal $\nu $
of $S.$

(vi) $\left( \mu \ast \nu \right) \wedge \left( \nu \ast \mu \right) \geq
\mu \wedge ^{\ast }\nu $ for every $\left( \in _{\gamma },\in _{\gamma }\vee
q_{\delta }\right) $-fuzzy left ideal $\mu $ and every $\left( \in _{\gamma
},\in _{\gamma }\vee q_{\delta }\right) $-fuzzy quasi- ideal $\nu $ of $S.$

(vii) $\left( \mu \ast \nu \right) \wedge \left( \nu \ast \mu \right) \geq
\mu \wedge ^{\ast }\nu $ for every $\left( \in _{\gamma },\in _{\gamma }\vee
q_{\delta }\right) $-fuzzy left ideal $\mu $ and every $\left( \in _{\gamma
},\in _{\gamma }\vee q_{\delta }\right) $-fuzzy bi-ideal $\nu $ of $S.$

(viii) $\left( \mu \ast \nu \right) \wedge \left( \nu \ast \mu \right) \geq
\mu \wedge ^{\ast }\nu $ for every $\left( \in _{\gamma },\in _{\gamma }\vee
q_{\delta }\right) $-fuzzy left ideal $\mu $ and every $\left( \in _{\gamma
},\in _{\gamma }\vee q_{\delta }\right) $-fuzzy generideal bi-ideal $\nu $
of $S.$

(ix) $\left( \mu \ast \nu \right) \wedge \left( \nu \ast \mu \right) \geq
\mu \wedge ^{\ast }\nu $ for every $\left( \in _{\gamma },\in _{\gamma }\vee
q_{\delta }\right) $-fuzzy quasi-ideals $\mu $ and $\nu $ of $S.$

(x) $\left( \mu \ast \nu \right) \wedge \left( \nu \ast \mu \right) \geq \mu
\wedge ^{\ast }\nu $ for every $\left( \in _{\gamma },\in _{\gamma }\vee
q_{\delta }\right) $-fuzzy quasi-ideal $\mu $ and every $\left( \in _{\gamma
},\in _{\gamma }\vee q_{\delta }\right) $-fuzzy left ideal $\nu $ of $S.$

(xi) $\left( \mu \ast \nu \right) \wedge \left( \nu \ast \mu \right) \geq
\mu \wedge ^{\ast }\nu $ for every $\left( \in _{\gamma },\in _{\gamma }\vee
q_{\delta }\right) $-fuzzy quasi-ideal $\mu $ and every $\left( \in _{\gamma
},\in _{\gamma }\vee q_{\delta }\right) $-fuzzy generalized bi-ideal $\nu $
of $S.$

(xii) $\left( \mu \ast \nu \right) \wedge \left( \nu \ast \mu \right) \geq
\mu \wedge ^{\ast }\nu $ for every $\left( \in _{\gamma },\in _{\gamma }\vee
q_{\delta }\right) $-fuzzy bi-ideals $\mu $ and $\nu $ of $S.$

(xiii) $\left( \mu \ast \nu \right) \wedge \left( \nu \ast \mu \right) \geq
\mu \wedge ^{\ast }\nu $ for every $\left( \in _{\gamma },\in _{\gamma }\vee
q_{\delta }\right) $-fuzzy bi-ideal $\mu $ and every $\left( \in _{\gamma
},\in _{\gamma }\vee q_{\delta }\right) $-fuzzy generalized bi-ideal $\nu $
of $S.$

(ixv) $\left( \mu \ast \nu \right) \wedge \left( \nu \ast \mu \right) \geq
\mu \wedge ^{\ast }\nu $ for every $\left( \in _{\gamma },\in _{\gamma }\vee
q_{\delta }\right) $-fuzzy generalized bi- ideals $\mu $ and $\nu $ of $S.$
\end{theorem}

\begin{proof}
$(i)\Longrightarrow (ixv)$ Let $\mu $ and $\nu $ be the $\left( \in _{\gamma
},\in _{\gamma }\vee q_{\delta }\right) $-fuzzy generalized bi-ideals of $S$
and $s\in S.$ Since $S$ is regular and intra-regular, so there exist $x,y$
and $z\in S,$ such that $s=\left( sx\right) s$ and $s=\left( ys^{2}\right)
z. $ Now, by using $(2),(3),(4)$ and $(1)$ we have,%
\begin{eqnarray*}
s &=&(sx)s=(((sx)s)x)((sx)s)=(((sx)((ys^{2})z))x)((sx)s) \\
&=&(((sx)((ys^{2})(ez)))x)((sx)s)=(((sx)((ze)(s^{2}y)))x)((sx)s) \\
&=&(((sx)(s^{2}((ze)y)))x)((sx)s)=((s^{2}((sx)((ze)y)))x)((sx)s) \\
&=&((x((sx)((ze)y))s^{2})((sx)s)=((x((ze)((sx)y)))s^{2})((sx)s) \\
&=&((x(z(sx))y))s^{2})((sx)s)=(((z(sx))(xy))s^{2})((sx)s) \\
&=&(((zx)((sx)y))s^{2})((sx)s)=(((zx)((((ys^{2})z)x)y))s^{2})((sx)s) \\
&=&(((zx)(((xz)(ys^{2}))ys^{2})((sx)s)=(((zx)(((s^{2}y)(zx))y))s^{2})((sx)s)
\\
&=&(((zx)((((zx)y)s^{2})y))s^{2})((sx)s)=(((zx)((ys^{2})((zx)y)))s^{2})((sx)s)
\\
&=&(((zx)((y(zx))(s^{2}y)))s^{2})((sx)s)=(((zx)(s^{2}((y(zx))y)))s^{2})((sx)s)
\\
&=&((s^{2}(zx((y(zx))y)))s^{2})((sx)s).
\end{eqnarray*}%
Then,%
\begin{eqnarray*}
\left( \mu \ast \nu \right) (s) &=&\left( (\mu \nu )(s)\vee \gamma \right)
\wedge \delta \\
&=&\left( \left( \bigvee_{s=bc}\{\mu (b)\wedge \nu (c)\}\right) \vee \gamma
\right) \wedge \delta \\
&\geq &\left( \{\mu ((s^{2}(zx((y(zx))y)))s^{2})\wedge \nu ((sx)s)\}\vee
\gamma \right) \wedge \delta \\
&=&\left( \{\left( \mu ((s^{2}(zx((y(zx))y)))s^{2})\vee \gamma \right)
\wedge \left( \nu ((sx)s)\vee \gamma \right) \}\vee \gamma \right) \wedge
\delta \\
&\geq &\left( \{\left( \mu ((s^{2})\wedge \delta \right) \wedge \left( \nu
(s)\wedge \delta \right) \}\vee \gamma \right) \wedge \delta \\
&=&\left( \{\mu (s^{2})\wedge \nu (s)\}\vee \gamma \right) \wedge \delta .
\end{eqnarray*}%
Now, by using $(1),(4),(2)$ and $(3)$ we have%
\begin{eqnarray*}
s^{2} &=&ss=((sx)s)((sx)s) \\
&=&(((sx)s)s)(sx)=((ss)(sx))(sx) \\
&=&(s((ss)x))(sx)=(((sx)((ss)x))s) \\
&=&((sx)((ss)(ex)))s=((sx)((xe)(ss)))s \\
&=&((sx)(s((xe)s)))s=(s((sx)((xe)s)))s.
\end{eqnarray*}%
Now, since $\mu $\ is an $\left( \in _{\gamma },\in _{\gamma }\vee q_{\delta
}\right) $-fuzzy generalized bi-ideals of $S.$ Therefore,%
\begin{eqnarray*}
\mu (s^{2})\vee \gamma &=&\mu ((s((sx)((xe)s)))s)\vee \gamma \\
&\geq &\mu (s)\wedge \mu (s)\wedge \delta \\
&=&\mu (s)\wedge \delta .
\end{eqnarray*}%
Thus, 
\begin{eqnarray*}
\left( \mu \ast \nu \right) (s) &\geq &\left( \{\mu (s^{2})\wedge \nu
(s)\}\vee \gamma \right) \wedge \delta \\
&=&\left( \{\left( \mu (s^{2})\vee \gamma \right) \wedge \nu (s)\}\vee
\gamma \right) \wedge \delta \\
&\geq &\left( \{\left( \mu (s)\wedge \delta \right) \wedge \nu (s)\}\vee
\gamma \right) \wedge \delta \\
&=&\left( \{\mu (s)\wedge \nu (s)\}\vee \gamma \right) \wedge \delta \\
&=&\left( \left( \mu \wedge \nu \right) (s)\vee \gamma \right) \wedge \delta
\\
&=&\left( \mu \wedge ^{\ast }\nu \right) (s).
\end{eqnarray*}%
Similarly, we can prove that $\left( \nu \ast \mu \right) (s)\geq \left( \mu
\wedge ^{\ast }\nu \right) (s).$ Hence, $\left( \mu \ast \nu \right) \wedge
\left( \nu \ast \mu \right) \geq \left( \mu \wedge ^{\ast }\nu \right) .$ $%
(ixv)\implies (xiii)\implies (xii)\implies (x)\implies (ix)\implies
(iii)\implies (ii),$ $(ixv)\implies (xi)\implies (x)$, $(ixv)\implies
(viii)\implies (iiv)\implies (xi)\implies (ii)$ and $(ixv)\implies
(v)\implies (iv)\implies (iii)\implies (ii)$ are obvious. $(ii)\implies (i).$
Let $\mu $ be an $\left( \in _{\gamma },\in _{\gamma }\vee q_{\delta
}\right) $-fuzzy right ideal and $\nu $\ be an $\left( \in _{\gamma },\in
_{\gamma }\vee q_{\delta }\right) $-fuzzy left ideal of $S.$ For $s\in S,$
we have%
\begin{eqnarray*}
\left( \mu \ast \nu \right) (s) &=&\left( (\mu \nu )(s)\vee \gamma \right)
\wedge \delta \\
&=&\left( \left( \bigvee_{s=yz}\{\mu (y)\wedge \nu (z)\}\right) \vee \gamma
\right) \wedge \delta \\
&=&\left( \left( \bigvee_{s=yz}\{\left( \mu (y)\wedge \delta \right) \wedge
\left( \nu (z)\wedge \delta \right) \}\right) \vee \gamma \right) \wedge
\delta \\
&\leq &\left( \left( \bigvee_{s=yz}\{\left( \mu (yz)\vee \gamma \right)
\wedge \left( \nu (yz)\vee \gamma \right) \}\right) \vee \gamma \right)
\wedge \delta \\
&=&\left( \left( \bigvee_{s=yz}\{\mu (yz)\wedge \nu (yz)\}\right) \vee
\gamma \right) \wedge \delta \\
&=&\left( \{\mu (s)\wedge \nu (s)\}\vee \gamma \right) \wedge \delta \\
&=&\left( \left( \mu \wedge \nu \right) (s)\vee \gamma \right) \wedge \delta
\\
&=&\left( \mu \wedge ^{\ast }\nu \right) (s).
\end{eqnarray*}%
Therefore, $\mu \ast \nu \leq \mu \wedge ^{\ast }\nu .$ By hypothesis, $\mu
\ast \nu \geq \mu \wedge ^{\ast }\nu .$ Thus, $\left( \mu \ast \nu \right)
=\left( \mu \wedge ^{\ast }\nu \right) .$ Hence, by Theorem \ref{TH4}, $S$
is regular. Also by hypothesis, $\mu \ast \nu \geq \mu \wedge ^{\ast }\nu .$
So, by Theorem \ref{TH6}, $S$ \ is intra-regular.
\end{proof}

\section{Conclusion}


\begin{thebibliography}{99}
\bibitem{221} S. Abdullah, M. Aslam, T. A. Khan and M. Naeem, A new type of
fuzzy normal subgroup and cosets, Journal of Intelligent and Fuzzy Systems,
2012, DOI 10.3233/IFS-2012-0612.

\bibitem{222} M. Aslam, S. Abdullah and N. Amin, Characterizations of gamma
LA-semigroups by generalized fuzzy gamma ideals, Int. Journal of Mathematics
and Statistics, 11 (2012), 29-50.

\bibitem{A222} S. Abdullah, M. A. Khan and M. Aslam, A new generalization of
fuzzy ideals in LA-semigroups, submitted

\bibitem{22} A. Bargiela and W. Pedrycz, Granular Computing, An
Introduction, in: The Kluwer Inter. Series in Engineering and Computer
Science, vol. 717, KluweAcademic Publishers, Boston, MA, ISBN:
1-4020-7273-2, 2003, p. xx+452.

\bibitem{9} S. K. Bhakat and P. Das, On the definition of a fuzzy subgroup,
Fuzzy Sets and Systems 51 (1992) 235 241.

\bibitem{12} S. K. Bhakat, $\left( \in \vee q\right) $-level subset, Fuzzy
Sets and Systems 103 (1999) 529 533.

\bibitem{11} S. K. Bhakat, $\left( \in ,\in \vee q\right) $-fuzzy normal,
quasinormal and maximal subgroups, Fuzzy Sets and Systems 112 (2000) 299 312.

\bibitem{13} S. K. Bhakat and P. Das, Fuzzy subrings and ideals redefined,
Fuzzy Sets and Systems 81 (1996) 383 393.

\bibitem{10} S. K. Bhakat and P. Das, $\left( \in ,\in \vee q\right) $-fuzzy
subgroups, Fuzzy Sets and Systems 80 (1996) 359 368.

\bibitem{21} A. H. Clifford and G.B. Preston, The Algebraic Theory of
Semigroups, Volume II, Ammerican Mathematical Socity, 1967.

\bibitem{14} B. Davvaz, $\left( \in ,\in \vee q\right) $-fuzzy subnearrings
and ideals, Soft Comput. 10 (2006) 206 211.

\bibitem{19} Y. B. Jun, Generalizations of $\left( \in ,\in \vee q\right) $%
-fuzzy subalgebras in BCK/BCI-algebras, Comput. Math. Appl. 58 (2009) 1383
1390.

\bibitem{15} Y. B. Jun and S. Z. Song, Generalized fuzzy interior ideals in
semigroups, Inform. Sci. 176, (2006), 3079 3093.

\bibitem{3} N. Kuroki, Fuzzy bi-ideals in semigroups, Comment. Math. Univ.
St. Pauli 28 (1979) 17 21.

\bibitem{4} N. Kuroki, On fuzzy ideals and fuzzy bi-ideals in semigroups,
Fuzzy Sets and Systems 5 (1981) 203 215.

\bibitem{16} O. Kazanci and S. Yamak, Generalized fuzzy bi-ideals of
semigroup, Soft Comput. 12 (2008) 1119 1124.

\bibitem{MNA} M. Khan and N. Ahmad, Characterization of left almost
semigroups by their ideals, 2 (3), (2010), 61 - 73.

\bibitem{Fuzz} M, Khan and M, N. A. Khan, On fuzzy abel Grassmann's
groupoids, Advanced in Fuzzy Mathematics, 5(3), (2010), 349-360.

\bibitem{qw} M.A. Kazim, M. Naseeruddin. On almost semigroups. The Alig.
Bull. Math., 1972, 2: 1 - 7.

\bibitem{RRM} R.A.R. Monzo. On the structure of abel Grassmann unions of
groups. International Journal of Algebra., 2010, 4:1261 - 1275.

\bibitem{QMK} Q. Mushtaq, M. S. Kamran, On LA-semigroups with weak
associative law. Scientific Khyber, 1989, 1: 69 - 71.

\bibitem{QMY} Q. Mushtaq, S. M. Yousuf. On LA-semigroups. The Alig. Bull.
Math., 1978, 8: 65 - 70.

\bibitem{QMMY} Q. Mushtaq, S. M. Yousuf. On LA-semigroup defined by a
commutative inverse semigroup. Math. Bech., 1988, 40:59 - 62.

\bibitem{MM} Q. Mushtaq, M. Khan. Ideals in left almost semigroups.
Proceedings of 4th International Pure Mathematics Conference, 2003, 65 - 77.

\bibitem{7} V. Murali, Fuzzy points of equivalent fuzzy subsets, Inform.
Sci. 158 (2004) 277 288.

\bibitem{24} X. Ma, J. Zhan and Y. B. Jun, New types of fuzzy ideals of
BCI-algebras, Neural Comp Appl., doi(2011): 10.1007/s00521-011-0558-x.

\bibitem{MN} M. Naseeruddin. Some studies in almost semigroups and flocks.
Ph.D. thesis. Aligarh Muslim University, Aligarh, India., 1970.

\bibitem{23} W. Pedrycz and F. Gomide, An Introduction to Fuzzy Sets, in:
Analysis and Design, with a Foreword by Lotfi A. Zadeh, Complex Adaptive
Syst. A Bradford Book, MIT Press, Cambridge, MA, ISBN: 0-262-16171-0, 1998,
p. xxiv+465.

\bibitem{8} P. M. Pu, Y. M. Liu, Fuzzy topology I, neighborhood structure of
a fuzzy point and Moore Smith convergence, J. Math. Anal. Appl. 76 (1980)
571 599.

\bibitem{2} A. Rosenfeld, Fuzzy subgroups, J. Math. Anal. Appl. 35 (1971)
512 517.

\bibitem{17} M. Shabir, Y. B. Jun, Y. Nawaz, Characterizations of regular
semigroups by $\left( \alpha ,\beta \right) $-fuzzy ideals, Comput. Math.
Appl. 59 (2010) 161 175.

\bibitem{26} M. Shabir, Y. B. Jun and Y. Nawaz, Semigroups characterized by $%
\left( \in ,\in \vee q_{k}\right) $-fuzzy ideals, Comput. Math. Appl. 60
(2010) 1473-1493.

\bibitem{1} L.A. Zadeh, Information and Control 8 (1965) 338 353.

\bibitem{25} J. Zhan and Y. Yin, New types of fuzzy ideal of near rings,
Neural Comp Appl. doi(2011):10.1007/s00521-011-0570-1.
\end{thebibliography}
\end{document}